\documentclass[11pt,reqno]{amsart}

\usepackage[margin=1.5in]{geometry}
\usepackage[colorinlistoftodos]{todonotes}

\newcounter{todocounter}

\usepackage[T1]{fontenc}
\usepackage{tgpagella}
\usepackage{mathpazo}

\usepackage[pagebackref]{hyperref} 
\usepackage{xcolor}
\usepackage{amsmath,amsthm}
\usepackage{amssymb}
\usepackage{graphicx}
\usepackage[mathscr]{eucal}
\usepackage{MnSymbol}
\usepackage[frame,ps,matrix,arrow,curve,rotate,all,2cell,tips]{xy}
\usepackage{epic,eepic}
\usepackage{enumerate}

\newtheorem{theorem}[equation]{Theorem}
\newtheorem{lemma}[equation]{Lemma}

\newtheorem{proposition}[equation]{Proposition}
\newtheorem{corollary}[equation]{Corollary}

\newtheorem{prop}[equation]{Proposition}

\theoremstyle{definition}
\newtheorem{definition}[equation]{Definition}
\newtheorem{example}[equation]{Example}
\newtheorem{remark}[equation]{Remark}

\newtheorem{convention}[equation]{Convention}

\numberwithin{equation}{subsection}

\renewcommand{\O}{\mathcal{O}}

\newcommand{\M}{\mathcal{M}}

\newcommand{\C}{\mathcal{C}}

\newcommand{\D}{\mathcal{D}}

\newcommand{\sI}{\mathscr{I}}
\newcommand{\sJ}{\mathscr{J}}

\renewcommand{\emptyset}{\varnothing}

\renewcommand{\tilde}[1]{\widetilde{#1}}

\DeclareMathOperator{\colim}{colim}

\DeclareMathOperator{\Sym}{Sym}

\DeclareMathOperator{\Ho}{Ho}

\newcommand{\po}{\ar@{}[dr]|(.7){\Searrow}}
\newcommand{\pb}{\ar@{}[dr]|(.3){\Nwarrow}}

\DeclareMathOperator{\map}{map}
\newcommand{\Ch}{\mathsf{Ch}}

\newcommand{\boxprod}{\mathbin\square}

\newcommand{\Sp}{\mathsf{Sp}}
\newcommand{\sigmatop}{\sigmat^{\mathrm{op}}}
\newcommand{\sigmat}{\Sigma_t}

\newcommand{\coprodover}[1]{\underset{#1}{\coprod}}
\newcommand{\dotover}[1]{\underset{#1}{\centerdot}}
\newcommand{\tensorover}[1]{\underset{#1}{\otimes}}
\newcommand{\timesover}[1]{\underset{#1}{\times}}
\newcommand{\algo}{\alg_{\sO}}
\newcommand{\sptothec}{\Sp^{\fC}}
\newcommand{\smallbinom}[2]
{\raisebox{.05cm}{\scalebox{0.8}{$\binom{#1}{#2}$}}}

\usepackage{tikz}
\usetikzlibrary{arrows,decorations.pathmorphing}
\usetikzlibrary{backgrounds,positioning}
\usetikzlibrary{fit,petri,shapes.misc}

\tikzset{auto}

\tikzset{empty/.style={circle,inner sep=0pt,minimum size=6mm}}
\tikzset{emptyvt/.style={circle,inner sep=0pt,minimum size=0mm}}

\tikzset{plain/.style={circle,draw,very thick,
inner sep=0pt,minimum size=6mm}}

\tikzset{fatplain/.style={rounded rectangle,draw,very thick,minimum size=6mm}}

\tikzset{bigplain/.style={rounded rectangle,draw,very thick,minimum size=.8cm}}

\tikzset{yellowvt/.style={circle,draw,fill=yellow,very thick,inner sep=0pt,minimum size=6mm}}

\tikzset{bluevt/.style={circle,draw,fill=blue!20,very thick,inner sep=0pt,minimum size=6mm}}

\tikzset{greenvt/.style={circle,draw,fill=green!30,very thick,inner sep=0pt,minimum size=6mm}}

\tikzset{redvt/.style={circle,draw,fill=red!30,very thick,inner sep=0pt,minimum size=6mm}}

\tikzset{arrow/.style={->,thick}}
\tikzset{dashedarrow/.style={->,dashed,thick}}
\tikzset{dottedarrow/.style={->,dotted,thick}}
\tikzset{mapto/.style={|->,thick}}

\tikzset{implies/.style={thick,double,double equal sign distance,-implies}}

\tikzset{line/.style={thick}}
\tikzset{dottedline/.style={dotted,thick}}
\tikzset{dashedline/.style={dashed,thick}}

\tikzset{inputleg/.style={<-,thick}}
\tikzset{outputleg/.style={->,thick}}
\tikzset{dottedinput/.style={<-,dotted,thick}}

\newcommand{\adjoint}{
\nicearrow\xymatrix{ \ar@<2pt>[r] & \ar@<2pt>[l]}}
\renewcommand{\hookrightarrow}{\nicexy{\ar@{^{(}->}[r] &}}
\newcommand{\nicearrow}{\SelectTips{cm}{10}}
\newcommand{\nicexy}{\nicearrow\xymatrix@C+10pt@R+10pt}

\newcommand{\pushout}{\ar@{}[dr]|(0.75){\Searrow}}
\newcommand{\drrpushout}{\ar@{}[drr]|(0.90){\Searrow}}

\newcommand{\comp}{\circ}

\newcommand{\ftilde}{\tilde{f}}

\renewcommand{\to}{\longrightarrow}

\newcommand{\bone}{\mathbf{1}}

\newcommand{\fB}{\mathfrak{B}}
\newcommand{\frakC}{\mathfrak{C}}
\newcommand{\fC}{\mathfrak{C}}

\newcommand{\asmod}{\mathsf{AsMod}}

\newcommand{\sg}{\mathsf{G}}

\newcommand{\sh}{\mathsf{H}}

\renewcommand{\sI}{\mathsf{I}}
\renewcommand{\sJ}{\mathsf{J}}

\newcommand{\sO}{\mathsf{O}}
\newcommand{\osuba}{\mathsf{O}_{A}}
\newcommand{\osubazero}{\mathsf{O}_{A}^0}
\newcommand{\osubaone}{\mathsf{O}_{A}^1}
\newcommand{\osubatwo}{\mathsf{O}_{A}^2}
\newcommand{\osubatminusone}{\mathsf{O}_{A}^{t-1}}
\newcommand{\osubat}{\mathsf{O}_{A}^{t}}
\newcommand{\osubainfinity}{\mathsf{O}_{A_{\infty}}}

\newcommand{\emptyo}{\varnothing_{\sO}}

\newcommand{\scrI}{\mathscr{I}}

\newcommand{\xtilde}{\widetilde{X}}

\newcommand{\ua}{\underline{a}}
\newcommand{\ub}{\underline{b}}
\newcommand{\uc}{\underline{c}}

\newcommand{\smallop}{{\scalebox{.5}{$\mathrm{op}$}}}
\newcommand{\cof}{{\scalebox{.5}{$\mathrm{cof}$}}}
\newcommand{\acof}{{\scalebox{.5}{$\mathrm{t.cof}$}}}

\newcommand{\clubcof}{(\clubsuit)_{\cof}}
\newcommand{\clubacof}{(\clubsuit)_{\acof}}

\newcommand{\cald}{\mathcal{D}}
\newcommand{\caldop}{\mathcal{D}^{\smallop}}

\newcommand{\calm}{\mathcal{M}}
\newcommand{\calmc}{\calm^{\fC}}

\newcommand{\calmg}{\calm^{\sg}}
\newcommand{\calmh}{\calm^{\sh}}

\newcommand{\set}{\mathsf{Set}}

\newcommand{\sset}{\mathsf{sSet}}

\newcommand{\symseq}{\mathsf{SymSeq}}
\newcommand{\symseqc}{\symseq_{\fC}}
\newcommand{\symseqcm}{\symseqc(\calm)}

\newcommand{\alg}{\mathsf{Alg}}
\newcommand{\Alg}{\mathsf{Alg}}

\newcommand{\sigmab}{\Sigma_{[\ub]}}
\newcommand{\sigmac}{\Sigma_{[\uc]}}

\newcommand{\sigmaop}{\Sigma^{\smallop}}
\newcommand{\sigmaopa}{\Sigma^{\smallop}_{\smallbra}}
\newcommand{\sigmaopb}{\Sigma^{\smallop}_{\smallbrb}}
\newcommand{\sigmaopc}{\Sigma^{\smallop}_{\smallbrc}}

\newcommand{\opc}{\mathsf{Op}^{\fC}}
\newcommand{\opcset}{\mathsf{Op}^{\fC}_{\mathsf{Set}}}
\newcommand{\operad}{\mathsf{Operad}}
\newcommand{\operadsigma}{\operad^{\Sigma}}
\newcommand{\operadsigmac}{\operad^{\sigmaofc}}
\newcommand{\operadsigmacm}{\operad^{\sigmaofc}_{\calm}}

\newcommand{\pofc}{\Sigma_{\frakC}}
\newcommand{\pofcop}
{\pofc^{\scalebox{.6}{$\mathrm{op}$}}}

\renewcommand{\pb}{\mathcal{P}(\fB)}

\newcommand{\smallprof}[1]
{\raisebox{.05cm}{\scalebox{0.8}{#1}}}

\newcommand{\cjbrbj}
{\smallprof{$\binom{c_j}{[\ub_j]}$}}

\newcommand{\singlecibi}
{\smallprof{$\binom{c_i}{\ub_i}$}}

\newcommand{\ccsingle}
{\smallprof{$\binom{c}{c}$}}

\newcommand{\singleciempty}
{\smallprof{$\binom{c_i}{\varnothing}$}}

\newcommand{\singledbra}
{\smallprof{$\binom{d}{[\ua]}$}}

\newcommand{\singledbrabrb}
{\smallprof{$\binom{d}{[\ua]; [\ub]}$}}
\newcommand{\singledbrabrbbrc}
{\smallprof{$\binom{d}{[\ua]; [\ub]; [\uc]}$}}

\newcommand{\singledbrtbbrc}
{\smallprof{$\binom{d}{[tb]; [\uc]}$}}

\newcommand{\singledbrtbc}
{\smallprof{$\binom{d}{[tb, \uc]}$}}

\newcommand{\singledbrac}
{\smallprof{$\binom{d}{[\ua,\uc]}$}}
\newcommand{\singledbrabc}
{\smallprof{$\binom{d}{[\ua,\ub,\uc]}$}}
\newcommand{\singledaprimecprime}
{\smallprof{$\binom{d}{\ua',\uc'}$}}
\newcommand{\singledaprimebprimecprime}
{\smallprof{$\binom{d}{\ua',\ub',\uc'}$}}

\newcommand{\singledbrb}
{\smallprof{$\binom{d}{[\ub]}$}}
\newcommand{\dub}
{\smallprof{$\binom{d}{\ub}$}}

\newcommand{\duc}
{\smallprof{$\binom{d}{\uc}$}}

\newcommand{\singledbrc}
{\smallprof{$\binom{d}{[\uc]}$}}

\newcommand{\singledbrtc}
{\smallprof{$\binom{d}{[tc]}$}}
\newcommand{\singledbrabrc}
{\smallprof{$\binom{d}{[\ua]; [\uc]}$}}

\newcommand{\singledempty}
{\smallprof{$\binom{d}{\varnothing}$}}

\newcommand{\smallbr}[1]
{\raisebox{.03cm}{\scalebox{0.5}{#1}}}

\newcommand{\smallbra}{\smallbr{$[\ua]$}}

\newcommand{\smallbrb}{\smallbr{$[\ub]$}}

\newcommand{\smallbrtbc}{\smallbr{$[tb,\uc]$}}

\newcommand{\smallbrc}{\smallbr{$[\uc]$}}

\newcommand{\sigmabra}{\Sigma_{\smallbr{$[\ua]$}}}
\newcommand{\sigmabraone}{\Sigma_{\smallbr{$[\ua_1]$}}}
\newcommand{\sigmabraj}{\Sigma_{\smallbr{$[\ua_j]$}}}
\newcommand{\sigmabram}{\Sigma_{\smallbr{$[\ua_m]$}}}

\newcommand{\sigmabrac}{\Sigma_{\smallbr{$[\ua,\uc]$}}}
\newcommand{\sigmabrabc}{\Sigma_{\smallbr{$[\ua,\ub,\uc]$}}}

\newcommand{\sigmabrb}{\Sigma_{\smallbr{$[\ub]$}}}

\newcommand{\sigmabrbj}{\Sigma_{\smallbr{$[\ub_j]$}}}
\newcommand{\sigmabrtbc}{\Sigma_{\smallbrtbc}}

\newcommand{\sigmabrc}{\Sigma_{\smallbr{$[\uc]$}}}

\newcommand{\sigmabraop}{\sigmabra^{\smallop}}
\newcommand{\sigmabrabcop}{\sigmabrabc^{\smallop}}

\newcommand{\sigmabracop}{\sigmabrac^{\smallop}}
\newcommand{\sigmabrbop}{\sigmabrb^{\smallop}}
\newcommand{\sigmabrbjop}{\sigmabrbj^{\smallop}}
\newcommand{\sigmabrtbcop}{\sigmabrtbc^{\smallop}}

\newcommand{\sigmabrcop}{\sigmabrc^{\smallop}}

\newcommand{\sigmabrcopd}{\sigmabrcop \times \{d\}}

\newcommand{\sigmaofc}{\pofc}
\newcommand{\sigmacop}{\pofcop}
\newcommand{\sigmacopc}{\sigmacop \times \fC}

\newcommand{\dbrch}{([\uc];d)}

\newcommand{\andspace}{\qquad\text{and}\qquad}

\renewcommand{\lim}{\mathsf{lim}\,}

\DeclareMathOperator{\Hom}{Hom}

\DeclareMathOperator{\id}{id}
\DeclareMathOperator{\Id}{Id}
\DeclareMathOperator{\Kan}{\mathsf{Kan}}

\DeclareMathOperator{\Ob}{Ob}

\begin{document}

\title{Bousfield Localization and Algebras over Colored Operads}

\author{David White}
\address{Denison University
\\ Granville, OH}
\email{david.white@denison.edu}

\author{Donald Yau}
\address{The Ohio State University at Newark \\ Newark, OH}
\email{dyau@math.osu.edu}

\begin{abstract}
We provide a very general approach to placing model structures and semi-model structures on algebras over symmetric colored operads. Our results require minimal hypotheses on the underlying model category $\M$, and these hypotheses vary depending on what is known about the colored operads in question. We obtain results for the classes of colored operad which are cofibrant as a symmetric collection, entrywise cofibrant, or arbitrary. As the hypothesis on the operad is weakened, the hypotheses on $\M$ must be strengthened. Via a careful development of the categorical algebra of colored operads we provide a unified framework which allows us to build (semi-)model structures for all three of these classes of colored operads. We then apply these results to provide conditions on $\M$, on the colored operad $O$, and on a class $\C$ of morphisms in $\M$ so that the left Bousfield localization of $\M$ with respect to $\C$ preserves $O$-algebras. Even the strongest version of our hypotheses on $\M$ is satisfied for model structures on simplicial sets, chain complexes over a field of characteristic zero, and symmetric spectra. We obtain results in these settings allowing us to place model structures on algebras over any colored operad, and to conclude that monoidal Bousfield localizations preserve such algebras.
\end{abstract}

\maketitle

\tableofcontents

\section{Introduction}

Modern algebraic topology has conclusively demonstrated the value of applying algebraic techniques to solve problems in homotopy theory. This has led to numerous results in stable homotopy theory (e.g. \cite{ekmm}) and, thanks to the generality of model categories, to homological algebra, algebraic geometry, (higher) category theory, equivariant homotopy theory, and even graph theory. Operads provide the means by which to encode algebraic structure in the necessary level of generality to recover all these examples, and operads have also found application in deformation theory and mathematical physics, in representation theory, in gauge theory and symplectic geometry, in graph cohomology, and in Goodwillie calculus. For a comprehensive overview, see \cite{fresse-book}.

In recent years, the importance of colored operads has become clear, e.g. in \cite{bm07}, \cite{jy2}, and \cite{batanin-berger}. Colored operads encode even more general algebraic structures, including the category of operads itself, other categories which encode algebraic structure (e.g. modular operads, higher operads, colored operads), morphisms between algebras over an operad, modules over an operad, other enriched categories, and diagrams in such categories. Colored operads have been applied in enriched category theory, factorization homology, higher category theory (leading to $\infty$-operads), and topological quantum field theories.

When studying operads and their algebras it is often advantageous to have model structures on these categories of algebras. For instance, in \cite{white-thesis} a theory is developed which obtains conditions under which left Bousfield localization preserves algebra structure when such categories of algebras possess appropriate (semi-)model structures. Such structures provide a powerful computational tool which has been crucial in many of the applications above. Our goal is to build (semi-)model structures on algebras over colored operads in the maximal possible generality, i.e., with as few hypotheses on the underlying model category as possible. For this reason we divide our focus between colored operads which are cofibrant, entrywise cofibrant, and arbitrary. We provide hypotheses under which these categories of algebras are model categories, and we provide weaker hypotheses so that they are semi-model categories, extending results of \cite{white-thesis} to the colored setting. We then apply these semi-model structures to prove results regarding preservation of algebraic structure by Bousfield localization.

After reviewing the necessary definitions and notation in Section \ref{sec:prelims}, we provide a careful development of the categorical algebra underlying the study of colored operads. This includes realizing the category of colored operads as a category of monoids for a particular monoidal product (which generalizes the circle product for operads) in Section \ref{sec:colored}, building the category of algebras over a colored operad in this setting, and producing a filtration \eqref{aoycolim} in Section \ref{sec:alg-over-colored} which can be used to transfer model structures to categories of algebras. This filtration generalizes the one found in \cite{harper-jpaa} and introduces a colored analogue for the symmetric sequence $\sO_A$ used therein. Filtrations of this sort have been studied by many authors in the setting of operads, but a careful treatment for the case of colored operads has not previously appeared.

In Section \ref{sec:O_A} we prove various homotopical properties for the colored symmetric sequence $\sO_A$, and in Section \ref{sec:model-on-algebras} we use our filtration to place model structures (and semi-model structures when hypotheses are relaxed) on categories of algebras over various classes of colored operads. In Section \ref{sec:preservation} we build on the work in \cite{white-localization} and provide general conditions so that left Bousfield localization preserves algebras over colored operads. Finally, in Section \ref{sec:applications} we provide numerous applications of these results to placing model structures on categories of algebras over any colored operad in simplicial sets, chain complexes over a field of characteristic zero, and to several model structures on symmetric spectra. In addition we prove results regarding preservation of colored operad algebras in these settings and we highlight future applications to ongoing research in equivariant stable homotopy theory, motivic homotopy theory, and higher categorical algebra including a new proof of the Breen-Baez-Dolan Stabilization Hypothesis.

\subsection*{Acknowledgments}

The authors are indebted to John E. Harper for numerous helpful conversations and to Luis Pereira for pointing out a mistake in an earlier version of this paper.

\section{Preliminaries}
\label{sec:prelims}

In this paper, $(\calm, \otimes, I, \Hom)$ will be a symmetric monoidal closed category with $\otimes$-unit $I$ and internal hom $\Hom$.  We assume $\calm$ has all small limits and colimits.  Its initial and terminal objects are denoted by $\varnothing$ and $*$, respectively.

At times we will also assume $\calm$ possesses a model structure that is compatible with the monoidal structure in a way we shall describe shortly. We will make it clear when we are assuming $\calm$ is a model category; much of the categorical algebra in this paper will not require a model structure on $\calm$.

\subsection{Monoidal Model Categories}

We assume the reader is familiar with basic facts about model categories as presented in \cite{hirschhorn} and \cite{hovey}. When we work with model categories they will most often be cofibrantly generated, i.e., there is a set $I$ of cofibrations and a set $J$ of trivial cofibrations (i.e. maps which are both cofibrations and weak equivalences) which permit the small object argument (with respect to some cardinal $\kappa$), and a map is a (trivial) fibration if and only if it satisfies the right lifting property with respect to all maps in $J$ (resp. $I$). This set $I$ is not to be confused with the monoidal unit, and the meaning of $I$ will be easy to infer from the context.

Let $I$-cell denote the class of transfinite compositions of pushouts of maps in $I$, and let $I$-cof denote retracts of such. In order to run the small object argument, we will assume the domains $K$ of the maps in $I$ (and $J$) are $\kappa$-small relative to $I$-cell (resp. $J$-cell), i.e., given a regular cardinal $\lambda \geq \kappa$ and any $\lambda$-sequence $X_0\to X_1\to \cdots$ formed of maps $X_\beta \to X_{\beta+1}$ in $I$-cell, then the map of sets
\[
\nicexy{\colim_{\beta < \lambda} \M\bigl(K,X_\beta\bigr) \ar[r] 
& \M\bigl(K,\colim_{\beta < \lambda} X_\beta\bigr)}
\]
is a bijection. An object is \emph{small} if there is some $\kappa$ for which it is $\kappa$-small. See Chapter 10 of \cite{hirschhorn} for a more thorough treatment of this material.

We must now discuss the interplay between the monoidal structure and the model structure which we will require in this paper. This definition is taken from 3.1 in \cite{ss}.

\begin{definition}[Monoidal Model Categories]
A symmetric monoidal closed category $\calm$ equipped with a model structure is called a \textbf{monoidal model category} if it satisfies the following axiom (known as the \textbf{pushout product axiom}): 

\begin{itemize}
\item Given any cofibrations $f:X_0\to X_1$ and $g:Y_0\to Y_1$, the pushout corner map
\[
\nicexy{
X_0\otimes Y_1 \coprod\limits_{X_0\otimes Y_0}X_1\otimes Y_0 
\ar[r]^-{f\boxprod g} & X_1\otimes Y_1}
\]
is a cofibration. If, in addition, either $f$ or $g$ is a weak equivalence then $f\boxprod g$ is a trivial cofibration.
\end{itemize}
\end{definition}

Note that the pushout product axiom is equivalent to the statement that $-\otimes-$ is a Quillen bifunctor.

\begin{remark}
\label{remark-check-pp-ax-on-gen}
If $\calm$ is cofibrantly generated, then Proposition 4.2.5 of \cite{hovey} shows that it is sufficient to check the pushout product axiom for $f$ and $g$ in the sets of generating (trivial) cofibrations.
\end{remark}

The monoidal adjunction of $\calm$ allows for an equivalent form of the pushout product axiom which we shall need (see Lemma 4.2.2 of \cite{hovey}).

\begin{remark}
\label{remark-fibration-version-pp-axiom}
The pushout product axiom holds if and only if the following statement holds:

\begin{itemize}
\item Given a cofibration $i:A\to B$ and a fibration $p:X\to Y$, the pullback corner map 
\[
\nicexy{
\Hom(B,X) \ar[r]^-{(i^*,p_*)} & \Hom(A,X)\timesover{\Hom(A,Y)} \Hom(B,Y)
}\] 
is a fibration, where $\Hom$ is the internal hom. Additionally, if either $i$ or $p$ is a weak equivalence then so is $(i^*,p_*)$.
\end{itemize}
\end{remark}

We will at times also need to assume an additional layer of compatibility between the monoidal structure and the model structure

\begin{definition}
\label{defn:cof-obj-flat}
Let $\calm$ be a monoidal model category. We say that \textit{cofibrant objects are flat} in $\calm$ if whenever an object $X$ is cofibrant and $f$ is a weak equivalence then $f\otimes X$ is a weak equivalence. 
\end{definition}

\subsection{Semi-Model Categories}

When attempting to study the homotopy theory of algebras over a colored operad, the usual method is to transfer a model structure from $\calm$ to this category of algebras along the free-forgetful adjunction (using Kan's Lifting Theorem \cite{hirschhorn} (11.3.2)). Unfortunately, it is often the case that one of the conditions for Kan's theorem cannot be checked fully, so that the resulting homotopical structure on the category of algebras is something less than a model category. This type of structure was first studied in \cite{hovey-monoidal} and \cite{spitzweck-thesis}, and later in published sources such as \cite{fresse} and \cite{fresse-book}.

\begin{definition}
\label{defn:semi-model-cat}
Assume there is an adjunction $F:\calm \rightleftarrows \D:U$ where $\calm$ is a cofibrantly generated model category, $\D$ is bicomplete, and $U$ preserves small colimits. 

We say that $\D$ is a \textbf{semi-model category} if $\D$ has three classes of morphisms called \emph{weak equivalences}, \emph{fibrations}, and \emph{cofibrations} such that the following axioms are satisfied.   A \emph{cofibrant} object $X$ means an object in $\D$ such that the map from the initial object of $\D$ to $X$ is a cofibration in $\D$.  Likewise, a \emph{fibrant} object is an object for which the map to the terminal object in $\D$ is a fibration in $\D$.

\begin{enumerate}
\item $U$ preserves fibrations and trivial fibrations ($=$ maps that are both weak equivalences and fibrations).
\item $\D$ satisfies the 2-out-of-3 axiom and the retract axiom of a model category.
\item Cofibrations in $\D$ have the left lifting property with respect to trivial fibrations. Trivial cofibrations ($=$ maps that are both weak equivalences and cofibrations) in $\D$ whose domain is cofibrant have the left lifting property with respect to fibrations.
\item Every map in $\D$ can be functorially factored into a cofibration followed by a trivial fibration. Every map in $\D$ whose domain is cofibrant can be functorially factored into a trivial cofibration followed by a fibration.
\item The initial object in $\D$ is cofibrant.
\item Fibrations and trivial fibrations are closed under pullback.
\end{enumerate}

$\D$ is said to be \textit{cofibrantly generated} if there are sets of morphisms $I'$ and $J'$ in $\D$ such that the following conditions are satisfied.
\begin{enumerate}
\item
Denote by $I'$-inj the class of maps that have the right lifting property with respect to maps in $I'$.  Then $I'$-inj is the class of trivial fibrations.
\item
$J'$-inj is the class of fibrations in $\D$.
\item
The domains of $I'$ are small relative to $I'$-cell.
\item
The domains of $J'$ are small relative to maps in $J'$-cell whose domain is sent by $U$ to a cofibrant object in $\calm$.
\end{enumerate}
\end{definition}

In practice the weak equivalences (resp. fibrations) are morphisms $f$ such that $U(f)$ is a weak equivalence (resp. fibration) in $\calm$, and the generating (trivial) cofibrations of $\D$ are maps of the form $F(I)$ and $F(J)$ where $I$ and $J$ are the generating (trivial) cofibrations of $\calm$.

Note that the only difference between a semi-model structure and a model structure is that one of the lifting properties and one of the factorization properties requires the domain of the map in question to be cofibrant. Because fibrant and cofibrant replacements are constructed via factorization, (4) of a semi-model category implies that every object has a cofibrant replacement and that cofibrant objects have fibrant replacements. So one could construct a fibrant replacement functor which first does cofibrant replacement and then does fibrant replacement. These functors behave as they would in the presence of a full model structure. 

The primary theorem we shall use to prove that our categories of interest possess semi-model structures is Theorem 3.3 in \cite{fresse}. Observe that Fresse requires slightly more of his semi-model categories than we do of ours (his axiom (1) is stronger than ours). The following theorem guarantees existence of a semi-model structure in the sense of Fresse, and hence in our sense as well.

\begin{theorem}[Semi-Model Category Existence Theorem]
\label{thm:fresse-semi-existence}
Assume that:\\
(*) for any pushout
\begin{align*}
\nicexy{F(X)\ar[r] \ar[d]_{F(i)}  & A \ar[d]^f \\ F(Y)\ar[r] & B} 
\end{align*}
where $A$ is a $F(\calm_{cof})$-cell complex (i.e. $\emptyset \to A$ is a transfinite composition of pushouts of maps of the form $F(h)$ where $h$ is a cofibration in $\calm$) then $U(f)$ is a (trivial) cofibration in $\calm$ whenever $i$ is a (trivial) cofibration in $\calm$.

Then $\D$ forms a cofibrantly generated semi-model category and $U:\D \to \calm$ maps cofibrations with cofibrant domains to cofibrations.
\end{theorem}

\section{Colored Operads}
\label{sec:colored}

In this section, we define colored operads as monoids with respect to a colored version of the circle product for operads.

\subsection{Colors and Profiles}

Here we recall from \cite{jy2} some notations regarding colors that are needed to talk about colored objects.

\begin{definition}[Colored Objects]
\label{def:profiles}
Fix a non-empty set $\fC$, whose elements are called \textbf{colors}.
\begin{enumerate}
\item
A \textbf{$\fC$-profile} is a finite sequence of elements in $\fC$, say,
\[
\uc = (c_1, \ldots, c_m) = c_{[1,m]}
\]
with each $c_i \in \fC$.  If $\fC$ is clear from the context, then we simply say \textbf{profile}. The empty $\fC$-profile is denoted $\emptyset$, which is not to be confused with the initial object in $\calm$.  Write $|\uc|=m$ for the \textbf{length} of a profile $\uc$.
\item
An object in the product category $\prod_{\fC} \calm = \calm^{\fC}$ is called a \textbf{$\fC$-colored object in $\calm$}, and similarly for a map of $\fC$-colored objects.  A typical $\fC$-colored object $X$ is also written as $\{X_a\}$ with $X_a \in \calm$ for each color $a \in \fC$.
\item
Suppose $X \in \calmc$ and $c \in \fC$.  Then $X$ is said to be \textbf{concentrated in the color $c$} if $X_d = \varnothing$ for all $c \not= d \in \fC$.
\item
Suppose $f : X \to Y \in \calm$ and $c \in \fC$.  Then $f$ is said to be \textbf{concentrated in the color $c$} if both $X$ and $Y$ are concentrated in the color $c$.
\end{enumerate}
\end{definition}

Next we define the colored version of a $\Sigma$-object, also known as a symmetric sequence.

\begin{definition}[Colored Symmetric Sequences]
\label{def:colored-sigma-object}
Fix a non-empty set $\fC$.
\begin{enumerate}
\item
If $\ua = (a_1,\ldots,a_m)$ and $\ub$ are $\fC$-profiles, then a \textbf{map} (or \textbf{left permutation}) $\sigma : \ua \to \ub$ is a permutation $\sigma \in \Sigma_{|\ua|}$ such that
\[
\sigma\ua = (a_{\sigma^{-1}(1)}, \ldots , a_{\sigma^{-1}(m)}) = \ub
\]
This necessarily implies $|\ua| = |\ub| = m$.
\item
The \textbf{groupoid of $\fC$-profiles}, with left permutations as the isomorphisms, is denoted by $\pofc$.  The opposite groupoid $\pofcop$ is regarded as the groupoid of $\fC$-profiles with \textbf{right permutations}
\[
\ua\sigma = (a_{\sigma(1)}, \ldots , a_{\sigma(m)})
\]
as isomorphisms.
\item
The \textbf{orbit} of a profile $\ua$ is denoted by $[\ua]$.  The maximal connected sub-groupoid of $\pofc$ containing $\ua$ is written as $\sigmabra$.  Its objects are the left permutations of $\ua$.  There is a decomposition
\begin{equation}
\label{pofcdecomp}
\pofc \cong \coprod_{[\ua] \in \pofc} \sigmabra,
\end{equation}
where there is one coproduct summand for each orbit $[\ua]$ of a $\fC$-profile.  By $[\ua] \in \pofc$ we mean that $[\ua]$ is an orbit in $\pofc$.
\item
Define the diagram category
\begin{equation}
\label{colored-symmtric-sequence}
\symseqcm = \calm^{\sigmacopc},
\end{equation}
whose objects are called \textbf{$\fC$-colored symmetric sequences}.  By the decomposition \eqref{pofcdecomp}, there is a decomposition
\begin{equation}
\label{symseqdecomp}
\symseqcm 
\cong 
\prod_{\dbrch \in \sigmacopc} \calm^{\sigmabrcopd},
\end{equation}
where $\sigmabrcopd \cong \sigmaopc$.  
\item
For $X \in \symseqcm$, we write
\begin{equation}
\label{sigmacopd-component}
X\singledbrc \in \calm^{\sigmabrcopd} \cong \calm^{\sigmabrcop}
\end{equation}
for its $\dbrch$-component.  For $(\uc;d) \in \sigmacopc$ (i.e., $\uc$ is a $\fC$-profile and $d \in \fC$), we write
\begin{equation}
\label{dc-component}
X\duc \in \calm
\end{equation}
for the value of $X$ at $(\uc;d)$.
\end{enumerate}
\end{definition}

\begin{remark}
\label{soneobject}
In the one-colored case (i.e., $\fC = \{*\}$), for each integer $n \geq 0$, there is a unique $\fC$-profile of length $n$, usually denoted by $[n]$.  We have $\Sigma_{[n]} = \Sigma_n$, the symmetric group $\Sigma_n$ regarded as a one-object groupoid.  So we have
\[
\pofc = \coprod_{n \geq 0} \Sigma_n = \Sigma
\andspace
\symseqcm = \calm^{\sigmacopc} = \calm^{\sigmaop}.
\]
In other words, one-colored symmetric sequences are symmetric sequences (also known as $\Sigma$-objects and collections) in the usual sense.
\end{remark}

From now on, assume that $\fC$ is a fixed non-empty set of colors, unless otherwise specified.

\begin{remark}
\label{rk:level-zero}
There is a fully faithful imbedding
\begin{equation}
\label{colored-object-imbedding}
\calm^{\fC} \to \symseqcm
\end{equation}
that sends a $\fC$-colored object $X = \{X_c\}_{c \in \fC}$ to the $\fC$-colored symmetric sequence with entries
\[
X\duc = \begin{cases}
X_d & \text{if $\uc = \varnothing$},\\
\varnothing & \text{if $\uc \not= \varnothing$},
\end{cases}
\]
where in the previous line the first (resp., second) $\varnothing$ denotes the initial object in $\calm$ (resp., the empty profile).
\end{remark}

\subsection{Colored Circle Product}

We will define $\fC$-colored operads as monoids with respect to the $\fC$-colored circle product.  To define the latter, we need the following definition.

\begin{definition}[Tensored over a Category]
\label{def:tensorover}
Suppose $\cald$ is a small groupoid, $X \in \calm^{\caldop}$, and $Y \in \calm^{\cald}$.  Define the object $X \otimes_{\cald} Y \in \calm$ as the colimit of the composite
\[
\nicexy{
\cald \ar[r]^-{\cong \Delta} 
& \caldop \times \cald \ar[r]^-{(X,Y)}
& \calm \times \calm \ar[r]^-{\otimes}
& \calm,
}\]
where the first map is the diagonal map followed by the isomorphism $\cald \otimes \cald \cong \caldop \times \cald$.
\end{definition}

We will mainly use the construction $\otimes_{\cald}$ when $\cald$ is the finite connected groupoid $\sigmabrc$ for some orbit $[\uc] \in \pofc$.

\begin{convention}
For an object $A \in \calm$, $A^{\otimes 0}$ is taken to mean $I$, the $\otimes$-unit in $\calm$.
\end{convention}

\begin{definition}[Colored Circle Product]
\label{def:colored-circle-product}
Suppose $X,Y  \in \symseqcm$, $d \in \fC$, $\uc = (c_1,\ldots,c_m) \in \pofc$, and $[\ub] \in \pofc$ is an orbit.
\begin{enumerate}
\item
Define the object
\[
Y^{\uc} \in \calm^{\pofcop} \cong \prod_{[\ub] \in \pofc} \calm^{\sigmabrbop}
\]
as having the $[\ub]$-component
\begin{equation}
\label{ytensorc}
Y^{\uc}([\ub]) 
=
\coprod_{\substack{\{[\ub_j] \in \pofc\}_{1 \leq j \leq m} \,\mathrm{s.t.} \\
[\ub] = [(\ub_1,\ldots,\ub_m)]}} 
\Kan^{\sigmabrbop} 
\left[\bigotimes_{j=1}^m Y \cjbrbj\right] 
\in \calm^{\sigmabrbop}.
\end{equation}
The left Kan extension in \eqref{ytensorc} is defined as
\[
\nicexy{
\prod_{j=1}^m \sigmabrbjop 
\ar[d]_-{\mathrm{concatenation}} 
\ar[rr]^-{\prod Y \binom{c_j}{-}} 
&&
\calm^{\times m} \ar[d]^-{\otimes}
\\
\sigmabrbop \ar[rr]_-{\Kan^{\sigmabrbop}\left[\otimes Y(\vdots)\right]}^-{\mathrm{left ~Kan~ extension}} 
&& \calm.
}\]
\item
By allowing left permutations of $\uc$ in \eqref{ytensorc}, we obtain
\[
Y^{[\uc]} \in \calm^{\pofcop \times \sigmabrc} \cong \prod_{[\ub] \in \pofc} \calm^{\sigmabrbop \times \sigmabrc}
\]
with components
\begin{equation}
\label{tensorbracket}
Y^{[\uc]}([\ub]) \in \calm^{\sigmabrbop \times \sigmabrc}.
\end{equation}
\item
Recall the product decomposition \eqref{symseqdecomp} of $\symseqcm$. 
The \textbf{$\fC$-colored circle product}
\[
X \circ Y \in \symseqcm
\]
is defined to have components
\begin{equation}
\label{xcircley}
(X \circ Y)\singledbrb 
= \coprod_{[\uc] \in \pofc} 
X\singledbrc \otimes_{\sigmabrc} Y^{[\uc]}([\ub]) \in \calm^{\sigmaopb \times \{d\}},
\end{equation}
where the coproduct is indexed by all the orbits in $\pofc$, as $d$ runs through $\fC$ and $[\ub]$ runs through all the orbits in $\pofc$.  The construction $\otimes_{\sigmabrc}$ was defined in Definition \ref{def:tensorover}.
\end{enumerate}
\end{definition}

\begin{remark}
In the one-colored case (i.e., $\fC = \{*\}$), the $\fC$-colored circle product is equivalent to the circle product of $\Sigma$-objects in \cite{rezk} (2.2.3).
\end{remark}

\begin{remark}
The appearance of the Kan extension in \eqref{ytensorc} may be explained as follows.  The object $Y^{\uc}([\ub])$ is supposed to have $\sigmaopb$-equivariance.  However, the tensor $\bigotimes_{j=1}^m Y\cjbrbj$ only has $(\prod \sigmabrbjop)$-equivariance, since $Y\cjbrbj$ is a $\sigmabrbjop$-equivariant object.  So we take the Kan extension to bump it up to a $\sigmaopb$-equivariant object.  Furthermore, in the one-colored case, this Kan extension is the usual copower operation $- \cdot_{\Sigma_{k_1} \times \cdots \times \Sigma_{k_m}} \Sigma_N$, where $N = k_1 + \cdots + k_m$.  The image of an object $X$ under this copower operation is, ignoring the $\Sigma_N$-equivariance, a coproduct of copies of $X$, one for each element in the quotient $\Sigma_N/(\Sigma_{k_1} \times \cdots \times \Sigma_{k_m})$.  The general colored case behaves similarly, as we will explain shortly.
\end{remark}

To explain $\Kan^{\sigmabrbop}$ explicitly, we need the following definition.

\begin{definition}
Suppose $\ua_j \in \sigmaofc$ for $1 \leq j \leq m$ and $\ua \in [(\ua_1,\ldots,\ua_m)]$.  An \textbf{order-preserving map}
\[
\sigma \in \sigmabra\bigl((\ua_1,\ldots,\ua_m); \ua\bigr)
\]
is a map such that, for each $1 \leq j \leq m$ and each color $d \in \fC$ that appears in $\ua_j$, the order of the images of these $d$'s under $\sigma$ is the same as in $\ua_j$.   Denote by
\[
\sigmabra'\bigl((\ua_1,\ldots,\ua_m); \ua\bigr)
\] 
the set of such order-preserving maps.
\end{definition}

\begin{example}
The set $\sigmabra' \bigl((\ua_1,\ldots,\ua_m); \ua\bigr)$ contains at least one element.  Moreover, if either
\begin{itemize}
\item
$m=1$, or
\item
the $\ua_j$'s do not have common colors (i.e., if $d \in \fC$ appears in some $\ua_j$, then $d$ does not appear in any $\ua_i$ for $i \not= j$),
\end{itemize}
then $\sigmabra' \bigl((\ua_1,\ldots,\ua_m); \ua\bigr)$ contains exactly one element.
\end{example}

The copower operation $- \cdot_{\Sigma_{k_1} \times \cdots \times \Sigma_{k_m}} \Sigma_N$ has the following colored analogue.

\begin{definition}
\label{def:explicit-kan}
Suppose $[\ua_j] \in \sigmaofc$ for $1 \leq j \leq m$, $[\ua] = [(\ua_1,\ldots,\ua_m)]$, and $X \in \calm^{\sigmabraone \times \cdots \times \sigmabram}$.  Define
\[
\xtilde \in \calm^{\sigmabra}
\]
as having the value
\begin{equation}
\label{kan-extension-formula}
\xtilde(\ua) 
= \coprod_{\{\ua_j \in \sigmabraj\}_{1 \leq j \leq m}}
\coprod_{\sigmabra'\left((\ua_1,\ldots,\ua_m); \ua\right)} 
X \left(\ua_1; \ldots; \ua_m\right) \in \calm
\end{equation}
for each object $\ua \in \sigmabra$.  To define the structure maps in $\xtilde$, suppose $\tau \in \sigmabra\left(\ua;\ub\right)$ for some profiles $\ua$ and $\ub$ in the orbit $[\ua]$, and suppose $\sigma \in \sigmabra'\bigl((\ua_1,\ldots,\ua_m); \ua\bigr)$.  Then $\tau\sigma \in \sigmabra\bigl((\ua_1,\ldots,\ua_m); \ub\bigr)$, but it may not be order-preserving.  However, there are unique permutations 
\begin{itemize}
\item
$\pi_j \in \sigmaofc(\ua_j;\ua_j)$  for $1 \leq j \leq m$ and 
\item
$\pi \in \sigmabra' \bigl((\ua_1,\ldots, \ua_m); \ub\bigr)$
\end{itemize}
such that the square
\begin{equation}
\label{sigma-tau-pi}
\nicexy{
(\ua_1,\ldots,\ua_m)
\ar[d]_-{\sigma}
\ar[r]^-{\{\pi_j\}}
& (\ua_1,\ldots, \ua_m)
\ar[d]^-{\pi}
\\
\ua \ar[r]^-{\tau} & \ub
}
\end{equation}
is commutative in $\sigmabra$.  More explicitly, for each $1\leq j \leq m$ and each color $d \in \fC$ that appears in $\ua_j$, say $k$ times, the images of these $k$ copies of $d$'s in $\ua$ have the same order as they do in $\ua_j$.  When restricted to these $k$ copies of $d$'s, $\tau$ permutes them in a certain way.  The permutation $\pi_j$ permutes the $k$ copies of $d$'s in $\ua_j$ exactly as $\tau$ does.  The map $\pi$ is defined as
\[
\pi = \tau \sigma \{\pi_j^{-1}\},
\]
which is order-preserving by construction.  Then we define the structure map
\[
\xtilde(\tau) : \xtilde(\ua) \to \xtilde(\ub)
\]
by sending the copy of $X \left(\ua_1; \ldots; \ua_m\right)$ in $\xtilde(\ua)$ corresponding to $\sigma$ to the copy of $X \left(\ua_1;\ldots; \ua_m\right)$ in $\xtilde(\ub)$ corresponding to $\pi$ via the structure map
\begin{equation}
\label{xpij}
\nicexy{
X(\ua_1;\ldots;\ua_m)
\ar[r]^-{X \{\pi_j\}}
& X(\ua_1;\ldots; \ua_m)
}
\end{equation}
in $X$.
\end{definition}

The next observation explains what $\Kan^{\sigmabraop}$ is.  To simplify the notations, we work with $\sigmabra$ instead.

\begin{proposition}
\label{kan-sigma}
Suppose $[\ua_j] \in \sigmaofc$ for $1 \leq j \leq m$, $[\ua] = [(\ua_1,\ldots,\ua_m)]$, and $X \in \calm^{\sigmabraone \times \cdots \times \sigmabram}$.  Define
\[
W = \Kan^{\sigmabra}(X) \in \calm^{\sigmabra}
\]
as the left Kan extension in:
\[
\nicexy{
\prod_{j=1}^m \sigmabraj
\ar[d]_-{\mathrm{concatenation}} 
\ar[rr]^-{X} 
&&
\calm \ar[d]^-{=}
\\
\sigmabra \ar[rr]_-{W}^-{\mathrm{left ~Kan~ extension}} 
&& \calm.
}\]
Then $W = \xtilde$.
\end{proposition}

\begin{proof}
First note that $\Kan^{\sigmabra}$ is the left adjoint
\[
\nicexy{
\calm^{\sigmabraone \times \cdots \times \sigmabram} 
\ar@<2pt>[r]^-{\Kan^{\sigmabra}} 
& \calm^{\sigmabra} \ar@<2pt>[l]
}\]
to the forgetful functor.  So we must show that $\xtilde$ has the universal property of the left adjoint.  Suppose $Y \in \calm^{\sigmabra}$ and
\[
f : X \to Y \in \calm^{\sigmabraone \times \cdots \times \sigmabram}.
\]
The desired unique extension
\[
\ftilde : \xtilde \to Y \in \calm^{\sigmabra}
\]
is defined as follows.  Suppose given an object $\ua \in [\ua]$, objects $\ua_j \in [\ua_j]$ for $1 \leq j \leq m$, and $\sigma \in \sigmabra'\bigl((\ua_1,\ldots,\ua_m); \ua\bigr)$.  Then the restriction of $\ftilde$ to $X(\ua_1;\ldots; \ua_m)_{\sigma}$ ($=$ the copy of $X(\ua_1;\ldots; \ua_m)$ in $\xtilde(\ua)$ corresponding to $\sigma$) is defined as the composition
\[
\nicexy{
X(\ua_1;\ldots; \ua_m)
\ar[r]^-{f}
& Y(\ua_1;\ldots; \ua_m)
\ar[r]^-{Y(\sigma)}
& Y(\ua).
}\]
That $\ftilde$ is a map in $\calm^{\sigmabra}$ follows from the following commutative diagram, in which we use the notations from Definition \ref{def:explicit-kan}:
\[
\nicexy{
\xtilde(\ua) \supseteq X(\ua_1;\ldots; \ua_m)_{\sigma}
\ar[d]_-{X\{\pi_j\}} \ar[r]^-{f}
& Y(\ua_1;\ldots; \ua_m)
\ar[d]_-{Y\{\pi_j\}} \ar[r]^-{Y(\sigma)}
& Y(\ua) \ar[d]^-{Y(\tau)}
\\
\xtilde(\ub) \supseteq X(\ua_1;\ldots; \ua_m)_{\pi}
\ar[r]^-{f}
& Y(\ua_1;\ldots; \ua_m)
\ar[r]^-{Y(\pi)}
& Y(\ub).
}\]
The left square is commutative because $f \in \calm^{\sigmabraone \times \cdots \times \sigmabram}$.  The right square is commutative because the square \eqref{sigma-tau-pi} is commutative.  The uniqueness of $\ftilde$ follows from the requirement that it extends $f$ and that it is $\sigmabra$-equivariant.
\end{proof}

So the upshot is that the Kan extension appearing in the $\fC$-colored circle product is given by the formulas in Definition \ref{def:explicit-kan}.

The following observation will be used to show that the $\fC$-colored circle product is associative.

\begin{lemma}
\label{circle-tensor}
Suppose $Y,Z \in \symseqcm$, and $[\ua], [\uc] \in \pofc$.  Then there is an isomorphism
\begin{equation}
\label{ycompzca}
(Y \comp Z)^{[\uc]}([\ua])
\cong
\coprod_{[\ub] \in \pofc}
Y^{[\uc]}([\ub]) \tensorover{\sigmabrb} Z^{[\ub]}([\ua])
\end{equation}
in $\calm^{ \sigmaopa \times \sigmabrc}$.
\end{lemma}

\begin{proof}
Denote by $W$ the right side of \eqref{ycompzca}.  Suppose $|\uc| = p$. For $\uc \in [\uc]$, there are isomorphisms in $\calm^{\sigmaopa}$ :
\[
\begin{split}
W(\uc)
&= \coprod_{[\ub] \in \pofc}
Y^{\uc}([\ub]) \tensorover{\sigmabrb} Z^{[\ub]}([\ua])
\\
&\cong \coprod_{\substack{\{[\ub_j] \in \pofc\}_{1\leq j \leq p}
\\ \{[\ua_j] \in \pofc\}_{1 \leq j \leq p} \mathrm{\,s.t.\,}
\\ [\ua] = [(\ua_1,\ldots,\ua_p)]}}
\Kan^{\sigmaopa} \left[ \bigotimes_{j=1}^p 
Y\cjbrbj \tensorover{\sigmabrbj} Z^{[\ub_j]}([\ua_j])
\right] \quad \text{(by \eqref{ytensorc})}
\\
&\buildrel * \over \cong  
\coprod_{\substack{\{[\ua_j] \in \pofc\}_{1 \leq j \leq p} \mathrm{\,s.t.\,}
\\ [\ua] = [(\ua_1,\ldots,\ua_p)]}}
\Kan^{\sigmaopa} \left[
\bigotimes_{j=1}^p \left(\coprod_{[\ub_j] \in \pofc} 
Y \cjbrbj \tensorover{\sigmabrbj} Z^{[\ub_j]}([\ua_j])\right)
\right]
\\
&= (Y \comp Z)^{\uc}([\ua]).
\end{split}
\]
The isomorphism $*$ follows from the fact that a map out of each of the two objects under consideration is equivalent to a map out of the other object in $\calm^{\sigmaopa}$.  Finally, observe that the above isomorphisms are compatible with the maps in  $\sigmabrc$.
\end{proof}

\begin{proposition}
\label{circle-product-monoidal}
With respect to $\circ$, $\symseqcm$ is a monoidal category.
\end{proposition}

\begin{proof}
The $\comp$-unit is the $\fC$-colored symmetric sequence $\scrI$ with entries
\begin{equation}
\label{circle-unit}
\scrI \duc = 
\begin{cases}
I & \text{if $\uc = d$},\\
\varnothing & \text{if $\uc \not= d$}
\end{cases}
\end{equation}
for $(\uc;d) \in \sigmacopc$.  To prove associativity, suppose $X,Y,Z \in \symseqcm$, $d \in \fC$, and $[\ua] \in \pofc$.  Then in $\calm^{\sigmaopa \times \{d\}}$ there are isomorphisms:
\[
\begin{split}
&\left[(X \comp Y) \comp Z\right] \singledbra\\
&= \coprod_{[\ub] \in \sigmaofc} (X \comp Y)\singledbrb \tensorover{\sigmab} Z^{[\ub]}([\ua])
\\
&= \coprod_{[\ub] \in \sigmaofc} 
\left[ \coprod_{[\uc] \in \sigmaofc} 
X\singledbrc \tensorover{\sigmac} Y^{[\uc]}([\ub]) \right]
 \tensorover{\sigmab} Z^{[\ub]}([\ua])
 \\
&\cong \coprod_{[\ub],\, [\uc] \in \sigmaofc} 
X \singledbrc \tensorover{\sigmac} \left[
Y^{[\uc]}([\ub]) \tensorover{\sigmab} Z^{[\ub]}([\ua]) \right]
\\
&\cong \coprod_{[\uc] \in \sigmaofc} X \singledbrc \tensorover{\sigmac}
\left[
\coprod_{[\ub] \in \sigmaofc} 
Y^{[\uc]}([\ub]) \tensorover{\sigmab} Z^{[\ub]}([\ua]) \right]
\\
&\cong \coprod_{[\uc] \in \sigmaofc}
X \singledbrc \tensorover{\sigmac} (Y \comp Z)^{[\uc]}([\ua])
\quad \text{(by \eqref{ycompzca})}
\\
&= \left[X \comp (Y \comp Z)\right] \singledbra.
\end{split}
\]
Since $d \in \fC$ and $[\ua] \in \sigmaofc$ are arbitrary, there is an isomorphism
\[
(X \comp Y) \comp Z
\cong
X \comp (Y \comp Z)
\]
in $\symseqcm$.
\end{proof}

\subsection{Colored Operads as Monoids}

\begin{definition}
\label{def:colored-operad}
For a non-empty set $\fC$ of colors, denote by
\[
\operadsigmacm
\]
or simply  $\operadsigmac$ the category of monoids \cite{maclane} (VII.3) in the monoidal category $(\symseqcm, \comp, \scrI)$.  An object in $\operadsigmac$ is called a \textbf{$\fC$-colored operad} in $\calm$.
\end{definition}

\begin{remark}
The $\Sigma$ in the notation $\operadsigmac$ is supposed to remind the reader that our colored operads have equivariant structures.  In the literature, a $\fC$-colored operad is sometimes called a symmetric multi-category with object set $\fC$.
\end{remark}

\begin{remark}
Unpacking Definition \ref{def:colored-operad}, a $\fC$-colored operad is equivalent to a triple $(\sO, \gamma, \bone)$ consisting of:
\begin{itemize}
\item
$\sO \in \symseqcm$,
\item
a \textbf{$\fC$-colored unit} map
\[
\nicexy{I \ar[r]^-{\bone_c} & \sO\ccsingle} \in \calm
\]
for each color $c \in \fC$, and
\item
\textbf{operadic composition}
\begin{equation}
\label{operadic-comp}
\nicexy{
\sO\duc \otimes \bigotimes\limits_{i=1}^m \sO\singlecibi \ar[r]^-{\gamma} 
& \sO\dub \in \calm
}
\end{equation}
for all $d \in \fC$, $\uc  = (c_1,\ldots,c_m) \in \sigmaofc$ with $m \geq 1$, and $\ub_i \in \sigmaofc$, where $\ub = (\ub_1,\ldots,\ub_m)$.
\end{itemize}
The triple $(\sO,\gamma,\bone)$ is required to satisfy some associativity, unity, and equivariance axioms, the details of which can be found in \cite{jy2} (11.14).   The detailed axioms in the one-colored case can also be found in \cite{may97}.  This way of expressing a $\fC$-colored operad is close to the way an operad was defined in \cite{may72}.  There are other equivalent ways to formulate the definition of a $\fC$-colored operad.

Intuitively, one should think of the component $\sO\duc \in \calm$ as the object of operations of the form,
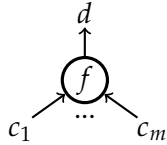
\begin{center}
\begin{tikzpicture}
\matrix[row sep=1cm, column sep=1cm]{
\node [plain, label=below:$...$] (f) {$f$};\\
};
\draw [outputleg] (f) to node[above=.1cm]{$d$} +(0,.7cm);
\draw [inputleg] (f) to node[below left=.1cm]{$c_1$} +(-.7cm,-.5cm);
\draw [inputleg] (f) to node[below right=.1cm]{$c_m$} +(.7cm,-.5cm);
\end{tikzpicture}
\end{center}
where $\uc = (c_1,\ldots,c_m)$ are the input colors (with $m=0$ allowed) and $d$ is the unique output color.  The symmetry in $\sO$ corresponds to permutations of the input colors.  The operadic composition $\gamma$ corresponds to the $2$-level tree:
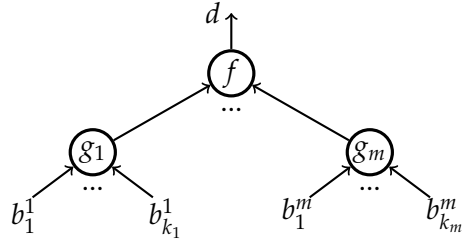
\begin{center}
\begin{tikzpicture}
\matrix[row sep=.1cm, column sep=1.2cm]{
& \node [plain, label=below:$...$] (f) {$f$}; &\\
\node [plain, label=below:$...$] (g1) {$g_1$}; &&
\node [plain, label=below:$...$] (gm) {$g_m$};\\
};
\draw [outputleg] (f) to node[at end]{$d$} +(0,.8cm);
\foreach \x in {g1,gm}
{
\draw [arrow] (\x) to (f);
}
\draw [inputleg] (g1) to node[below left=.1cm]{$b^1_1$} +(-.8cm,-.6cm);
\draw [inputleg] (g1) to node[below right=.1cm]{$b^1_{k_1}$} +(.8cm,-.6cm);
\draw [inputleg] (gm) to node[below left=.1cm]{$b^m_1$} +(-.8cm,-.6cm);
\draw [inputleg] (gm) to node[below right=.1cm]{$b^m_{k_m}$} +(.8cm,-.6cm);
\end{tikzpicture}
\end{center}
Here $f$ must have non-empty inputs (i.e., $m \geq 1$), but each $k_i$ may be $0$.  In particular, the inputs of this $2$-level tree are the concatenation of the lists $\left(b^i_1, \ldots , b^i_{k_i}\right)$ for $1 \leq i \leq m$.  Associativity of the operadic composition takes the form of a $3$-level tree.  The $c$-colored unit map corresponds to the tree $\uparrow_c$ with no vertices.  Detailed discussion of graphs, and in particular trees, related to operads can be found in \cite{jy2} (Part I).  Using such trees, it is possible to show that a $\fC$-colored operad is exactly an algebra over a certain monad associated to the pasting scheme of unital trees \cite{jy2} (11.16).  There is also a description of $\fC$-colored operads based on certain $\comp_i$-operations.
\end{remark}

\begin{remark}
In the one-colored case (i.e., $\fC = \{*\}$), write $\operadsigma$ for $\operadsigmac$, whose objects are called \textbf{$1$-colored operads}.  In this case we write $\sO(n)$ for the $([n]; *)$-component of $\sO \in \operadsigma$, where $[n]$ is the orbit of the $\{*\}$-profile consisting of $n$ copies of $*$ (this orbit has only one object).  Our notion of a $1$-colored operad agrees with the notion of an operad in, e.g., \cite{may97} and \cite{harper-jpaa}.  Note that even for $1$-colored operads, our definition is slightly more general than the one in \cite{mss} (II.1.2) because ours has the $0$-component $\sO(0)$, corresponding to the empty $\{*\}$-profile.  In general the purpose of the $0$-component (whether in the one-colored or the general colored cases) is to encode units in $\sO$-algebras, e.g., units in associative algebras.  Also note that in \cite{may72}, where an operad was first defined in the topological setting, the $0$-component was required to be a point.
\end{remark}

\begin{definition}
Suppose $n \geq 0$.  A $\fC$-colored symmetric sequence $X$ is said to be \textbf{concentrated in arity $n$} if
\[
|\uc| \not= n 
\quad \Longrightarrow\quad 
X\duc = \varnothing ~ \text{for all $d \in \fC$.}
\]
\end{definition}

\begin{example} 
\begin{enumerate}
\item
A $\fC$-colored symmetric sequence concentrated in arity $0$ is precisely a $\fC$-colored object via the fully faithful imbedding in Remark \ref{rk:level-zero}.  In the $\fC$-colored circle product $X \comp Y$ \eqref{xcircley}, if $Y$ is concentrated in arity $0$, then so is $X \comp Y$ because, by \eqref{ytensorc}, 
\[
\ub \not= \varnothing \quad \Longrightarrow \quad Y^{\uc}([\ub]) = \varnothing
\]
for all $\uc$.  So if $\sO$ is a $\fC$-colored operad, then the functor
\begin{equation}
\label{o-comp-monad}
\sO \comp - : \calmc \to \calmc
\end{equation}
defines a monad \cite{maclane} (VI.1)  whose monadic multiplication and unit are induced by the multiplication $\sO \comp \sO \to \sO$ and the unit $\scrI \to \sO$, respectively.
\item
A $\fC$-colored operad concentrated in arity $1$ is also called a \textbf{ring with several objects}.  Note that a $\fC$-colored operad $\sO$ concentrated in arity $1$ is exactly a small category with object set $\fC$ enriched in $\calm$.  In this case, the non-trivial operadic compositions correspond to the categorical compositions.  Restricting further to the $1$-colored case $(\fC = \{*\})$, a $1$-colored operad concentrated in arity $1$ is precisely a monoid in $\calm$.
\end{enumerate}
\end{example}

\section{Algebras over Colored Operads}
\label{sec:alg-over-colored}

In this section, we define algebras over a colored operad and study their categorical properties.  The main result of this section is the filtration in \eqref{aoycolim} for the pushout of an $\sO$-algebra against a free map.  This filtration is a key component in establishing the desired (semi-)model structures on the category of $\sO$-algebras.

As before $(\calm, \otimes, I, \Hom)$ is a symmetric monoidal closed category with all small limits and colimits.  A model structure on $\calm$ is \emph{not} needed yet.

\subsection{Definition and Examples}
\label{subsec:def-colored-algebra}

Fix a non-empty set $\fC$ of colors.

\begin{definition}
\label{colored-operad-algebra}
Suppose $\sO$ is a $\fC$-colored operad.  The category of algebras over the monad \cite{maclane} (VI.2)
\[
\sO \comp - : \calmc \to \calmc
\]
in \eqref{o-comp-monad} is denoted by $\alg(\sO; \calm)$ or simply $\alg(\sO)$, whose objects are called \textbf{$\sO$-algebras} (in $\calm$).
\end{definition}

There are several equivalent ways to formulate the definition of an $\sO$-algebra.  To describe it more explicitly using the $\fC$-colored circle product, we use the following construction.

\begin{definition}
\label{def:asubc}
Suppose $A = \{A_c\}_{c\in \fC} \in \calm^{\fC}$ is a $\fC$-colored object.  For $\uc = (c_1,\ldots,c_n) \in \pofc$ and associated orbit $[\uc]$, define the object
\begin{equation}
\label{asubc}
A_{\uc} = \bigotimes_{i=1}^n A_{c_i} = A_{c_1} \otimes \cdots \otimes A_{c_n} \in \calm
\end{equation}
and the diagram $A_{\smallbrc} \in \calm^{\sigmabrc}$ with values
\begin{equation}
\label{asubbrc}
A_{\smallbrc}(\uc') = A_{\uc'}
\end{equation}
for each $\uc' \in [\uc]$.  All the structure maps in the diagram $A_{\smallbrc}$ are given by permuting the factors in $A_{\uc}$.
\end{definition}

\begin{remark}[Unwrapping $\sO$-Algebras]
From the definition of the monad $\sO \comp -$, an $\sO$-algebra $A$ has a structure map $\mu : \sO \comp A \to A \in \calmc$.  For each color $d \in \fC$, the $d$-colored entry of $\sO \comp A$ is
\begin{equation}
\label{o-comp-a-d}
(\sO \comp A)_d = \coprod_{[\uc] \in \sigmaofc} 
\sO\singledbrc \tensorover{\sigmabrc} A_{\smallbrc}.
\end{equation}
So the $d$-colored entry of the structure map $\mu$ consists of maps
\[
\nicexy{
\sO\singledbrc \tensorover{\sigmabrc} A_{\smallbrc}
\ar[r]^-{\mu}
& A_d \in \calm
}\]
for all orbits $[\uc] \in \sigmaofc$.  The $\otimes_{\sigmabrc}$ here means that we can unpack $\mu$ further into maps
\begin{equation}
\label{algebra-map-unpack}
\nicexy{
\sO\duc \otimes A_{\uc}
\ar[r]^-{\mu}
& A_d \in \calm
}
\end{equation}
for all $d \in \fC$ and all objects $\uc \in \sigmaofc$.  Then an $\sO$-algebra is equivalent to a $\fC$-colored object $A$ together with structure maps \eqref{algebra-map-unpack} that are associative, unital, and equivariant in an appropriate sense, the details of which can be found in \cite{jy2} (13.37).  The detailed axioms in the $1$-colored case can also be found in \cite{may97}.  Note that when $\uc = \varnothing$, the map \eqref{algebra-map-unpack} takes the form
\begin{equation}
\label{d-colored-units}
\nicexy{\sO \singledempty \ar[r]^-{\mu} & A_d}
\end{equation}
for $d \in \fC$.  In practice this $0$-component of the structure map gives $A$ the structure of $d$-colored units.  For example, in a unital associative algebra, the unit arises from the $0$-component of the structure map.
\end{remark}

Some examples of colored operads and their algebras follow.

\begin{example}[Initial and Terminal Colored Operads]
Suppose $\fC$ is a non-empty set of colors.
\begin{enumerate}
\item
The \textbf{initial $\fC$-colored operad} is the object $\scrI$ in \eqref{circle-unit}, whose $c$-colored unit is the identity map for each color $c \in \fC$.  Its operadic composition is given by the isomorphism $I \otimes I \cong I$.
\item
The \textbf{terminal $\fC$-colored operad} is the object in which every entry is the terminal object $*$ in $\calm$.
\end{enumerate}
\end{example}

\begin{example}[Free Operadic Algebras]
\label{ex:free-algebra}
Suppose $\sO$ is a $\fC$-colored operad.  
\begin{enumerate}
\item
There is an adjoint pair
\begin{equation}
\label{free-algebra-adjoint}
\nicexy{
\calmc \ar@<2pt>[r]^-{\sO \comp -} 
& \alg(\sO) \ar@<2pt>[l]
}
\end{equation}
in which the right adjoint is the forgetful functor.  So for a $\fC$-colored object $A$, the object $\sO \comp A$ has the canonical structure of an $\sO$-algebra, called the \textbf{free $\sO$-algebra of $A$}.  In particular, free $\sO$-algebras always exist.
\item
The initial object $\emptyset$ in $\calmc$ consists of the initial object in $\calm$ in each entry.  Since $\sO \comp -$ is a left adjoint, it preserves colimits and, in particular, the initial object.  So the image $\sO \comp \emptyset$ is the \textbf{initial $\sO$-algebra}, denoted $\emptyset_{\sO}$.  It follows from \eqref{o-comp-a-d} that, for each color $d \in \fC$, its $d$-colored entry is
\begin{equation}
\label{initial-o-algebra}
(\emptyset_{\sO})_d = (\sO \comp \emptyset)_d = \sO\singledempty.
\end{equation}
Its $\sO$-algebra structure map, in the form \eqref{algebra-map-unpack},
\[
\nicexy@R-10pt{
\sO\duc \otimes \bigotimes\limits_{i=1}^m (\emptyset_{\sO})_{c_i} 
\ar[d]_-{=} \ar[r] 
&
(\emptyset_{\sO})_d \ar[d]^-{=}
\\
\sO\duc \otimes \bigotimes\limits_{i=1}^m \sO\singleciempty \ar[r] 
&
\sO\singledempty,
}\]
is the operadic composition \eqref{operadic-comp} of $\sO$ with $\ub_i = \emptyset$ for $1 \leq i \leq m$.  For an $\sO$-algebra $A$, the unique $\sO$-algebra map $\emptyset_{\sO} \to A$ has the map $\sO\singledempty \to A_d$ in \eqref{d-colored-units} as its $d$-colored entry.
\end{enumerate}
\end{example}

\begin{example}[$\fC$-Colored Operads as Operadic Algebras]
\label{ex:operad-of-operad}
For each non-empty set of colors $\fC$, there exist an $[\Ob(\pofcop) \times \fC]$-colored operad $\opc$ and an isomorphism
\begin{equation}
\label{colored-op-algebras}
\operadsigmac \cong \alg(\opc).
\end{equation}
So $\fC$-colored operads are equivalent to algebras over the $[\Ob(\pofcop) \times \fC]$-colored operad $\opc$.  This is a special case of \cite{jy2} (14.4), which describes any category of generalized props (of which $\operadsigmac$ is an example) as a category of algebras over some colored operad.  Together with Example \ref{ex:free-algebra}, it follows that \textbf{free $\fC$-colored operads} ($=$ free $\opc$-algebras) always exist.  The colored operad $\opc$ is entry-wise a coproduct of copies of the $\otimes$-unit $I$.    In fact, its construction begins with an $[\Ob(\pofcop) \times \fC]$-colored operad $\opcset$ in the symmetric monoidal category of sets and Cartesian products.  There is a strong symmetric monoidal functor
\[
\set \to \calm, \quad S \longmapsto \coprod_S I.
\]
The colored operad $\opc$ is the entry-wise image of $\opcset$ under this strong symmetric monoidal functor.  Therefore, if $\calm$ has a model structure in which $I$ is cofibrant, then $\opc$ is entry-wise cofibrant.  In fact, when $I$ is cofibrant, a careful inspection of $\opc$ shows that its image $\opc \in \symseqcm$ is cofibrant.
\end{example}

\begin{example}[Diagrams of Algebras]
Suppose $\sO$ is a $\fC$-colored operad.  
\begin{enumerate}
\item
There is a $(\fC \bigsqcup \fC)$-colored operad $\sO_{\bullet \to \bullet}$ whose algebras are diagrams $f : A \to B$, in which $A$ and $B$ are $\sO$-algebras and $f$ is a map of $\sO$-algebras.  It can be constructed as a quotient of a free $(\fC \bigsqcup \fC)$-colored operad, the details of which can be found in \cite{fmy} (2.10).  
\item
Similarly, for any small category $\cald$, there exist a $(\coprod_{\Ob(\cald)} \fC)$-colored operad $\sO_{\cald}$ and an isomorphism
\[
\alg(\sO)^{\cald} \cong \alg(\sO_{\cald}),
\]
where the left side is the category of $\cald$-shaped diagrams in $\alg(\sO)$.
\end{enumerate}
\end{example}

\begin{example}[Monoid-Modules]
There is a $2$-colored operad $\asmod$ whose algebras are pairs $(A,M)$, where
\begin{itemize}
\item
$A$ is a monoid in $\calm$ \cite{maclane} (VII.3) and
\item
$M$ is a left $A$-module \cite{maclane} (VII.4).
\end{itemize}
Two colors are needed for such pairs because one color is needed for each of $A$ and $M$.  The $2$-colored operad $\asmod$ can be described as a quotient of a free $2$-colored operad, the details of which can be found in \cite{fmy} (2.11).
\end{example}

\subsection{Limits and Colimits of Colored Operadic Algebras}

Recall the free-forgetful adjoint pair
\[
\nicexy{
\calmc \ar@<2pt>[r]^-{\sO \comp -} 
& \alg(\sO) \ar@<2pt>[l]
}\]
in \eqref{free-algebra-adjoint} for a $\fC$-colored operad.

\begin{proposition}
\label{algebra-bicomplete}
Suppose $\sO$ is a $\fC$-colored operad.  Then the category $\alg(\sO)$ has all small limits and colimits, with reflexive coequalizers and filtered colimits preserved and created by the forgetful functor $\alg(\sO) \to \calmc$.
\end{proposition}

\begin{proof}
By definition $\alg(\sO)$ is the category of algebras over the monad $\sO \comp -$ in the product category $\calmc$, which has all small (co)limits.  So the existence of limits in $\alg(\sO)$ follows from \cite{borceux} (4.3.1).  In each color, the left adjoint $\sO \comp -$ is a coproduct of coinvariants (over finite connected groupoids) of finite tensor products \eqref{o-comp-a-d}.  This implies that $\alg(\sO)$ has filtered colimits and reflexive coequalizers, which are preserved and created by the forgetful functor.  A general colimit in $\alg(\sO)$ can then be constructed as a reflexive coequalizer using a well-known procedure, used in, e.g.,  \cite{rezk} (2.3.5), \cite{ekmm} (II.7.4), and \cite{fresse} (I.4.4-I.4.6).
\end{proof}

\subsection{Filtration for Pushouts of Colored Operadic Algebras}

\begin{definition}
Suppose $X \in \symseqcm$, $d \in \fC$, and $[\ua], [\ub], [\uc]$ are orbits in $\sigmaofc$.  
\begin{enumerate}
\item
Define the diagram
\begin{equation}
\label{xdac}
X \singledbrabrc \in \calm^{\sigmabraop \times \sigmabrcop \times \{d\}}
\end{equation}
as having the objects
\[
X \singledbrabrc(\ua'; \uc') 
= X \singledaprimecprime \in \calm
\]
for $\ua' \in [\ua]$ and $\uc' \in [\uc]$ and the structure maps of $X$.
\item
Likewise, define the diagram
\begin{equation}
\label{xdabc}
X \singledbrabrbbrc \in \calm^{\sigmabraop \times \sigmabrbop \times \sigmabrcop \times \{d\}}
\end{equation}
as having the objects
\[
X \singledbrabrbbrc(\ua'; \ub'; \uc') 
= X \singledaprimebprimecprime \in \calm
\]
for $\ua' \in [\ua]$, $\ub' \in [\ub]$, and $\uc' \in [\uc]$ and the structure maps of $X$.
\end{enumerate}
\end{definition}

\begin{remark}
\label{rk:xdac-restriction}
In other words, $X \singledbrabrc$ is the restriction of the component
\[
X \singledbrac \in \calm^{\sigmabracop \times \{d\}}
\]
of $X \in \symseqcm$ via the inclusion
\[
\nicexy{\sigmabraop \times \sigmabrcop \ar[r] & 
\sigmabracop.}
\]
This construction will be used below for a $\fC$-colored operad.  Similarly, $X \singledbrabrbbrc$ is the restriction of the $\singledbrabc$-component of $X$ via the inclusion
\[
\nicexy{\sigmabraop \times \sigmabrbop \times \sigmabrcop \ar[r] & 
\sigmabrabcop.
}\]
\end{remark}

\begin{definition}[$\sO_A$ for $\sO$-algebras]
\label{oaalgebra}
Suppose $\sO$ is a $\fC$-colored operad and $A \in \alg(\sO)$.  Define $\sO_A \in \symseqcm$ as follows.  For $d \in \fC$ and orbit $[\uc] \in \sigmaofc$, define the component
\begin{equation}
\label{oadc}
\sO_A\singledbrc \in \calm^{\sigmabrcop \times \{d\}}
\end{equation}
as the reflexive coequalizer of the diagram
\begin{equation}
\label{o-sub-a-coequal}
\nicexy{
\coprod\limits_{[\ua] \in \sigmaofc} 
\sO\singledbrabrc
\tensorover{\sigmabra} 
(\sO \circ A)_{\smallbra}
\ar@<-3pt>[d]_-{d_0} 
\ar@<3pt>[d]^-{d_1}
\\
\coprod\limits_{[\ua] \in \sigmaofc} 
\sO\singledbrabrc 
\tensorover{\sigmabra} A_{\smallbra}
\ar@<-1pc>@/_2pc/[u]
}
\end{equation}
with
\begin{itemize}
\item
the coequalizer taken in $\calm^{\sigmabrcop \times \{d\}}$, 
\item
$d_0$ induced by the operadic composition on $\sO$,
\item
$d_1$ induced by the $\sO$-algebra action on $A$, and 
\item
the common section induced by $A \cong \scrI \circ A \to \sO \circ A$.
\end{itemize}
\end{definition}

\begin{proposition}
\label{aplusoofy}
Suppose $\sO$ is a $\fC$-colored operad, $A \in \alg(\sO)$, and $Y \in \calmc$.   Then the $\sO$-algebra
\[
A \coprod (\sO \circ Y)
\]
has the following entries.  For each color $d \in \fC$, there is a natural isomorphism
\begin{equation}
\label{aoyd}
\left[A \coprod (\sO \circ Y)\right]_d 
\cong
\coprod_{[\ub] \in \sigmaofc}  
\left[\sO_A \singledbrb
\tensorover{\sigmabrb} 
Y_{\smallbrb}\right]
= \left(\sO_A \comp Y\right)_d 
\end{equation}
in $\calm$.
\end{proposition}

\begin{proof}
Since $\sO \comp - $ is a left adjoint, it sends a coproduct in $\calmc$ to a coproduct in $\alg(\sO)$.  Using this fact, we first compute the $d$-colored entry of
\[
[\sO \circ A] \coprod [\sO \circ Y]
\]
in $\calm$:
\begin{equation}
\label{oaoyd}
\begin{split}
\left([\sO \circ A] \coprod [\sO \circ Y]\right)_d
&\cong \left[\sO \circ (A \coprod Y)\right]_d 
\\
&= \coprod_{[\uc] \in \pofc} 
\sO \singledbrc \tensorover{\sigmabrc} 
\left(A \coprod Y\right)_{\smallbrc}
\\
&\cong \coprod_{[\ub] \in \sigmaofc} 
\left(
\coprod_{[\ua]\in\pofc} 
\sO \singledbrabrb \tensorover{\sigmabra} A_{\smallbra}
\right) \tensorover{\sigmabrb}
Y_{\smallbrb}.
\end{split}
\end{equation}
Notice the notation changed from $[\uc]$ to $[\ub]$ as the former splits as $[\uc] = [\ua,\ub]$.  Now replace $A$ in \eqref{oaoyd} with $\sO \circ A$ to obtain:
\begin{equation}
\label{ooaoyd}
\begin{split}
&\left(\left[\sO \circ \sO \circ A\right] 
\coprod \left[\sO \circ Y\right]\right)_d
\\
&\cong \coprod_{[\ub] \in \sigmaofc} 
\left(
\coprod_{[\ua]\in\pofc} 
\sO \singledbrabrb \tensorover{\sigmabra} (\sO \comp A)_{\smallbra}
\right) \tensorover{\sigmabrb}
Y_{\smallbrb}.
\end{split}
\end{equation}
Since $A$ is an algebra over the monad $\sO \comp -$, it is isomorphic to the reflexive coequalizer
\medskip
\begin{equation}
\label{algebra=coequalizer}
\colim\left(
\nicexy{
\sO \circ \sO \circ A 
\ar@<3pt>[r]^-{d_0}
\ar@<-3pt>[r]_-{d_1} & 
\sO \circ A \ar@/_2pc/[l]
}\right) \in \alg(\sO)
\end{equation}
by \cite{borceux} (4.3.3).  So there is an isomorphism
\[ 
\begin{split}
&A \coprod (\sO \circ Y) \\
&\cong
\colim\left(
\nicexy{
[\sO \circ \sO \circ A] \coprod [\sO \circ Y] 
\ar@<3pt>[r]^-{d_0}
\ar@<-3pt>[r]_-{d_1} & 
[\sO \circ A] \coprod [\sO \circ Y]
\ar@/_2pc/[l]
}\right),
\end{split}
\]
in $\alg(\sO)$, where the last reflexive coequalizer can be computed color-wise in $\calm$ by Proposition \ref{algebra-bicomplete}.  Now restrict to a typical $d$-colored entry using \eqref{oaoyd}, \eqref{ooaoyd}, and the definition of $\sO_A$ \eqref{oadc} to obtain the desired isomorphism \eqref{aoyd}.
\end{proof}

\begin{remark}
\label{lower-left-corner}
In the previous Proposition, if $Y$ is concentrated at a single color $c \in \fC$ (so $Y_b = \varnothing$ whenever $b \not= c$), then \eqref{aoyd} becomes
\[
\left[A \coprod (\sO \circ Y)\right]_d 
\cong
\coprod_{t \geq 0}  
\left[\sO_A \singledbrtc
\tensorover{\Sigma_t} 
Y^{\otimes t}\right]
= \left(\sO_A \comp Y\right)_d 
\]
where $tc = (c,\ldots,c)$ has $t$ copies of $c$.
\end{remark}

The following definition appeared in \cite{em06} (section 12) and \cite{harper-jpaa} (7.10)

\begin{definition}[$Q$-Construction]
\label{def:q-construction}
Suppose $i : X \to Y \in \calmc$ is concentrated at a single color $c \in \fC$ (so $X_b = Y_b = \varnothing$ whenever $b \not= c$) and $t \geq 1$.  For $0 \leq q \leq t$, define
\[
Q_q^t = Q_q^{[tc]} \in \M^{\Sigma_t}
\]
as follows.
\begin{itemize}
\item
$Q^{[tc]}_0 = X^{\otimes t}$.
\item
$Q^{[tc]}_t = Y^{\otimes t}$.
\item
For $0 < q < t$ there is a pushout in $\calm^{\Sigma_t}$:
\begin{equation}
\label{inductive-q-one-colored}
\nicexy{
\Sigma_t \dotover{\Sigma_{t-q} \times \Sigma_q} 
\left[ X^{\otimes (t-q)} 
\otimes Q_{q-1}^{[qc]}\right] 
\ar[d]_-{(\id,i_*)} \ar[r]
&
Q^{[tc]}_{q-1} \ar[d]
\\
\Sigma_t \dotover{\Sigma_{t-q} \times \Sigma_q} 
\left[ X^{\otimes (t-q)} 
\otimes 
Y^{\otimes q}\right] \ar[r] 
&
Q^{[tc]}_q.
}
\end{equation}
\end{itemize}
\end{definition}

The following observation is the colored analogue of \cite{harper-jpaa} (7.12) and also appeared in \cite{em06} (section 12).

\begin{proposition}
\label{colored_jpaa7.12}
Suppose $\sO$ is a $\fC$-colored operad, $A \in \alg(\sO)$, $i : X \to Y \in \calmc$ is concentrated at a single color $c \in \fC$, and
\begin{equation}
\label{algebra-pushout}
\nicexy{
\sO \circ X \ar[d]_-{\id \circ i} \ar[r]^-{f} 
& A \ar[d]^-{j}
\\
\sO \circ Y \ar[r] 
& A \coprod_{\sO \circ X} (\sO \circ Y)
}
\end{equation}
is a pushout in $\alg(\sO)$.  Then there is a natural isomorphism
\begin{equation}
\label{aoycolim}
A \coprod_{\sO \circ X} (\sO \circ Y)
 \cong 
\colim\left(
\nicexy{
A_0 \ar[r]^-{j_1} & A_1 \ar[r]^-{j_2} & A_2 \ar[r]^-{j_3} & \cdots
}
\right)
\end{equation}
in $\calmc$ such that the following statements hold.
\begin{itemize}
\item
$A_0 = A$.
\item
For each color $d \in \fC$ and $t \geq 1$, the $d$-colored entry of $A_t$ is inductively defined as the pushout in $\calm$
\begin{equation}
\label{one-colored-jt-pushout}
\nicexy{
\sO_A \singledbrtc
\tensorover{\Sigma_t} 
Q^{t}_{t-1}
\ar[d]_-{\id \tensorover{\Sigma_t} i^{\boxprod t}} 
\ar[r]^-{f^{t-1}_*} 
& 
(A_{t-1})_d \ar[d]^-{j_t}
\\
\sO_A \singledbrtc
\tensorover{\Sigma_t} 
Y^{\otimes t}
\ar[r]_-{\xi_{t}} 
& 
(A_t)_d.
}
\end{equation}
with $f^{t-1}_*$ induced by $f$ and $tc = (c,\ldots,c)$ with $t$ copies of $c$.
\end{itemize}
\end{proposition}

\begin{proof}
The plan is to show that the underlying $\fC$-colored object of the pushout $A \coprod_{\sO \comp X} (\sO \comp Y)$ can be computed as the sequential colimit in \eqref{aoycolim}.  There are several steps.
\begin{enumerate}
\item
Note that the  pushout $ A \coprod_{\sO \circ X} (\sO \circ Y) \in \alg(\sO)$ is isomorphic to the reflexive coequalizer
\begin{equation}
\label{pushoutcoequalizer}
\colim\left(
\nicexy{
A \coprod (\sO \circ X) \coprod (\sO \circ Y) 
\ar@<3pt>[r]^-{f_*} 
\ar@<-3pt>[r]_-{i_*} 
&
A \coprod (\sO \circ Y)
\ar@/_2pc/[l]
}\right).
\end{equation}
By Proposition \ref{algebra-bicomplete} this reflexive coequalizer can be computed color-wise in $\calm$ .  This reflexive coequalizer in $\calmc$ is characterized by the following universal properties.
\begin{enumerate}
\item
It receives a map from $A \coprod (\sO \circ Y)$ in $\calmc$.
\item
The pre-compositions of this map with the maps $f_*$ and $i_*$ are equal.
\item
It is initial with respect to the above properties.
\end{enumerate}
We will show that the sequential colimit in \eqref{aoycolim} has these universal properties.
\item
Proposition \ref{aplusoofy} says that the underlying $\fC$-colored object of $A \coprod (\sO \circ Y) \in \alg(\sO)$ is naturally isomorphic to a sequential colimit
\begin{equation}
\label{aoycolimitb}
A \coprod (\sO \circ Y) 
\cong 
\colim\left(
\nicexy{
B_0 \ar[r]^-{l_1} & B_1 \ar[r]^-{l_2} & B_2 \ar[r]^-{l_3} & \cdots
}
\right)
\end{equation}
in $\calmc$ such that the following statements hold.
\begin{itemize}
\item
$B_0 = A$ (by \eqref{o-sub-a-coequal} and \eqref{algebra=coequalizer}).
\item
For each color $d \in \fC$ and $t \geq 1$, there is a pushout in $\calm$:
\begin{equation}
\label{btispushout}
\nicexy{
\varnothing \ar[d] \ar[r] 
& 
(B_{t-1})_d \ar[d]^-{l_t}
\\
\sO_A \singledbrtc
\tensorover{\Sigma_t} 
Y^{\otimes t}
\ar[r]_-{\zeta_{t}} 
& (B_t)_d.
}
\end{equation}
\end{itemize}
The reason is that by Proposition \ref{aplusoofy} and Remark \ref{lower-left-corner}, for each color $d \in \fC$, we have
\[
\left[A \coprod (\sO \circ Y)\right]_d 
\cong
\coprod_{t \geq 0}  
\left[\sO_A \singledbrtc
\tensorover{\Sigma_t} 
Y^{\otimes t}\right]
= \left(\sO_A \comp Y\right)_d 
\]
Instead of writing it as a coproduct over all $t \geq 0$, we may also write it using the pushouts over $\varnothing$ as above.

Note that the lower left corners in the pushout squares \eqref{one-colored-jt-pushout} and \eqref{btispushout} are the same, namely,
\begin{equation}
\label{lowerleftcorner}
\sO_A \singledbrtc
\tensorover{\Sigma_t} 
Y^{\otimes t}
\end{equation}
Furthermore, there is a compatible sequence of maps from the pushout square \eqref{btispushout} to the pushout square \eqref{one-colored-jt-pushout} for $t \geq 1$ that is the identity map in the lower left corners \eqref{lowerleftcorner}.  This determines a map
\begin{equation}
\label{mapfromaplusoy}
\nicexy{
A \coprod (\sO \circ Y) \ar[r]^-{\pi} 
& \colim_k A_k \in \calmc.
}
\end{equation}
\item
We want to check that the sequential colimit $\colim_k A_k$ \eqref{aoycolim} in $\calmc$ has the universal properties of the reflexive coequalizer \eqref{pushoutcoequalizer} when computed in $\calmc$.
\begin{enumerate}
\item
The map from $A \coprod (\sO \circ Y)$ is $\pi$ \eqref{mapfromaplusoy}.
\item
The equality $\pi f_* = \pi i_*$ follows from the pushout definition \eqref{one-colored-jt-pushout} of $A_t$ in each color.
\item
To see that $(\colim_k A_k, \pi)$ is initial with respect to the above two properties, suppose given a map
\[
\nicexy{
A \coprod (\sO \circ Y) \cong \colim_k B_k \ar[r]^-{\varphi} 
& W \in \calmc
}\]
such that
\begin{equation}
\label{phifstar}
\varphi f_* = \varphi i_*.
\end{equation}
We want to show that $\varphi$ factors through $\pi$ uniquely, i.e., that there is a unique map $\psi : \colim_k A_k \to W$ as in 
\[
\nicexy{
A \coprod (\sO \circ Y) \cong \colim B_k 
\ar[d]_-{\varphi} \ar[r]^-{\pi} 
& \colim A_k \ar[dl]^-{\psi}\\
W&
}\]
such that $\varphi = \psi \pi$.  Let $\varphi_k : B_k \to W$ be the restriction of $\varphi$ to $B_k$.  To define the map $\psi$, it suffices to define a compatible sequence of maps $\psi_k : A_k \to W$ such that the diagram
\begin{equation}
\label{phikpsik}
\nicexy@C+10pt{
B_k \ar[d]_-{\varphi_k} \ar[r]^-{\pi_k} 
& A_k \ar[dl]^-{\psi_k}
\\
W&
}
\end{equation}
commutes for each $k \geq 0$.  Since $B_0 = A = A_0$, we are forced to define $\psi_0 = \varphi_0 : A_0 \to W$.  

Inductively, suppose we have defined compatible maps $\psi_k$ for $k < t$.  To define $\psi_t$, it is enough to define it in the typical $d$-colored entry.  The solid-arrow diagram
\[
\nicexy@C+10pt{
\sO_A \singledbrtc
\tensorover{\Sigma_t} 
Q^{t}_{t-1}
\ar[d]_-{\id \tensorover{\Sigma_t} i^{\boxprod t}} 
\ar[r]^-{f^{t-1}_*} 
& 
(A_{t-1})_d \ar[d]^-{j_t} 
\ar@/^2pc/[ddr]^-{\psi_{t-1}} &
\\
\sO_A \singledbrtc
\tensorover{\Sigma_t} 
Y^{\otimes t}
\ar[r]|-{\xi_{t}} \ar[d]_-{\zeta_t}
& 
(A_t)_d \ar@{.>}[dr]|-{\psi_t} &
\\
(B_t)_d \ar[rr]^-{\varphi_t} \ar[ur]|-{\pi_t}
&& W_d
}\]
in $\calm$ is commutative by \eqref{phifstar}, and \eqref{phikpsik}.  By the universal property of the pushout, there is a unique induced map
\[
\psi_t : (A_t)_d \to W_d
\]
such that
\begin{equation}
\label{psit}
\psi_{t-1} = \psi_t j_t \andspace
\varphi_t \zeta_t = \psi_t \xi_t.
\end{equation}
To see that
\[
\varphi_t = \psi_t \pi_t : (B_t)_d \to W_d,
\]
note that there is an  isomorphism
\[
(B_{t})_d 
\cong
(B_{t-1})_d \coprod 
\sO_A \singledbrtc
\tensorover{\Sigma_t} 
Y^{\otimes t}
\]
for each $t \geq 1$.  The restrictions of $\varphi_t$ and $\psi_t \pi_t$ to $(B_{t-1})_d$ coincide by the inductive construction of $\psi_t$.  So it is enough to see that their pre-compositions with $\zeta_t$ coincide as well.  This holds by the second equality in \eqref{psit} and $\xi_t = \pi_t \zeta_t$.  This defines the map $\psi$.

By construction we have $\varphi = \psi \pi$.  The uniqueness of $\psi$ follows from the pushout definition of the $A_t$.
\end{enumerate}
\end{enumerate}
\end{proof}

\section{More Properties of $\sO_A$}
\label{sec:O_A}

For now $(\calm,\otimes,I,\Hom)$ is still a symmetric monoidal closed category with all small limits and colimits.  This section contains some technical results that we will need to equip the category of algebras over a colored operad with a model structure or at least a semi-model structure.

\subsection{Recovering $\sO$ and $A$}

Recall $\sO_A \in \symseqcm$ for an $\sO$-algebra $A$ in Definition \ref{oaalgebra}.

\begin{proposition}
\label{o-sub-empty}
Suppose $\sO$ is a $\fC$-colored operad, and $\varnothing$ is the initial $\sO$-algebra.  Then there is an isomorphism
\begin{equation}
\label{oemptyiso}
\sO_{\varnothing} \cong \sO
\end{equation}
in $\symseqcm$.
\end{proposition}

\begin{proof}
Proposition \ref{aplusoofy} with $A = \varnothing$ gives the isomorphism
\[
(\sO \circ Y)_d 
\cong
\coprod_{[\ub] \in \sigmaofc}  
\left[\sO_{\varnothing} \singledbrb
\tensorover{\sigmabrb} 
Y_{\smallbrb}\right].
\]
Since this holds for all $Y \in \calmc$, the formula \eqref{o-comp-a-d} for $(\sO \comp Y)_d$ implies the desired isomorphism.
\end{proof}

Next we observe that we can also recover $A$ from $\sO_A$ by taking the $0$-components.

\begin{proposition}
\label{osuba-empty}
Suppose $\sO$ is a $\fC$-colored operad, $A \in \alg(\sO)$, and $d \in \fC$.  Then there is a natural isomorphism
\[
\sO_A\singledempty \cong A_d
\]
in $\calm$.
\end{proposition}

\begin{proof}
By Definition \ref{oaalgebra} $\sO_A\singledempty$ is the reflexive coequalizer of the diagram
\begin{equation}
\label{adreflexive}
\nicexy{
(\sO \comp \sO \comp A)_d =
\coprod\limits_{[\ua] \in \sigmaofc} 
\sO\singledbra
\tensorover{\sigmabra} 
(\sO \circ A)_{\smallbra}
\ar@<-3pt>[d]_-{d_0} 
\ar@<3pt>[d]^-{d_1}
\\
(\sO \comp A)_d = 
\coprod\limits_{[\ua] \in \sigmaofc} 
\sO\singledbra 
\tensorover{\sigmabra} A_{\smallbra}
\ar@<-1pc>@/_2pc/[u]
}
\end{equation}
But as mentioned in \eqref{algebra=coequalizer}, $A \in \alg(\sO)$ is naturally  isomorphism to the reflexive coequalizer of the diagram
\bigskip
\[
\nicexy{
\sO \circ \sO \circ A 
\ar@<3pt>[r]^-{d_0}
\ar@<-3pt>[r]_-{d_1} & 
\sO \circ A \ar@/_2pc/[l]
}\]
in $\alg(\sO)$, which can be computed color-wise in $\calm$ by Proposition \ref{algebra-bicomplete}.  So the reflexive coequalizer of \eqref{adreflexive} is isomorphic to $A_d$.
\end{proof}

\subsection{Coproduct with Free Algebras}

The next observation is the colored analogue of \cite{harper-gnt} (5.31) that we will need to use later.

\begin{proposition}
\label{colored-hh-5.31}
Suppose $\sO$ is a $\fC$-colored operad, $A \in \alg(\sO)$, $Y \in \calmc$, $d \in \fC$, and $[\uc] \in \sigmaofc$.  Consider the coproduct $A \coprod (\sO \comp Y) \in \alg(\sO)$ and the object
\[
\sO_{A \coprod (\sO \comp Y)} \in \symseqcm.
\]
Then there is a natural isomorphism
\begin{equation}
\label{o-sub-a-plus-oy}
\sO_{A \coprod (\sO \comp Y)} \singledbrc 
\cong
\coprod_{[\ua] \in \sigmaofc}
\left[
\sO_A\singledbrabrc \tensorover{\sigmabra} Y_{[\ua]}
\right]
\end{equation}
in $\calm^{\sigmabrcop \times \{d\}}$.
\end{proposition}

\begin{proof}
Suppose $Z \in \calmc$.  We will compute each entry of
\[
\left[A \coprod (\sO \circ Y) \coprod (\sO \comp Z)\right] \in \alg(\sO)
\]
in two different ways and compare them.  Using \eqref{aoyd} with $A$ and $Y$ replaced by $A \coprod (\sO \comp Y)$ and $Z$, respectively, there is an isomorphism
\begin{equation}
\label{aoyozd}
\left[A \coprod (\sO \circ Y) \coprod (\sO \comp Z)\right]_d 
\cong
\coprod_{[\uc] \in \sigmaofc}  
\left[\sO_{A\coprod (\sO\comp Y)} \singledbrc\right]
\tensorover{\sigmabrc} 
Z_{\smallbrc}.
\end{equation}
On the other hand, there are isomorphisms:
\begin{equation}
\label{aoyzd}
\begin{split}
&\left[A \coprod (\sO \circ Y) \coprod (\sO \comp Z)\right]_d
\\
&\cong \left[A \coprod \sO \comp (Y \coprod Z)\right]_d
\\
&\cong
\coprod_{[\ub] \in \sigmaofc}  
\left[\sO_{A} \singledbrb
\tensorover{\sigmabrb} 
(Y \coprod Z)_{\smallbrb}\right]
\quad \text{(by \eqref{aoyd})} 
\\
&\cong
\coprod_{[\uc] \in \sigmaofc}
\left[
\coprod_{[\ua] \in \sigmaofc}
\sO_A \singledbrabrc \tensorover{\sigmabra} Y_{\smallbra}\right]
\tensorover{\sigmabrc} Z_{\smallbrc}.
\end{split}
\end{equation}
Since \eqref{aoyozd} and \eqref{aoyzd} hold for all $Z \in \calmc$, the desired isomorphism \eqref{o-sub-a-plus-oy} follows.
\end{proof}

\begin{remark}
\label{oaplusoy-one-colored-y}
If $Y \in \calmc$ is concentrated at a single color $b \in \fC$, then \eqref{o-sub-a-plus-oy} becomes
\[
\sO_{A \coprod (\sO \comp Y)} \singledbrc 
\cong
\coprod_{t \geq 0}
\left[
\sO_A \singledbrtbbrc \tensorover{\Sigma_t} Y^{\otimes t}
\right]
\]
where $tb = (b,\ldots,b)$ with $t$ copies of $b$.
\end{remark}

\begin{corollary}
\label{colored-hh-5.31b}
Suppose $\sO$ is a $\fC$-colored operad, $Y \in \calmc$, $d \in \fC$, and $[\uc] \in \sigmaofc$.  Then there is a natural isomorphism
\begin{equation}
\label{o-sub-oy}
\sO_{\sO \comp Y} \singledbrc 
\cong
\coprod_{[\ua] \in \sigmaofc}
\left[
\sO\singledbrabrc \tensorover{\sigmabra} Y_{[\ua]}
\right]
\end{equation}
in $\calm^{\sigmabrcop \times \{d\}}$.
\end{corollary}

\begin{proof}
This follows from the isomorphism \eqref{o-sub-a-plus-oy} with $A = \varnothing$ (the initial $\sO$-algebra) and the isomorphism $\sO \cong \sO_{\varnothing}$ in $\symseqcm$ \eqref{oemptyiso}.
\end{proof}

\subsection{Pushout of a Free Map}

The following observation, which is the colored version of \cite{harper-gnt} (5.7), will be used in the next result.

\begin{lemma}
\label{osubdc}
Suppose $\sO$ is a $\fC$-colored operad, $d \in \fC$, and $[\uc] \in \sigmaofc$.  Then the functor
\[
\sO_{(-)}\singledbrc : \alg(\sO) \to \calm^{\sigmabrcop \times \{d\}}
\]
preserves reflexive coequalizers and filtered colimits.
\end{lemma}

\begin{proof}
This follows from Proposition \ref{algebra-bicomplete} (that reflexive coequalizers and filtered colimits in $\alg(\sO)$ can be computed color-wise in $\calm$), the definition \eqref{o-sub-a-coequal} (of $\sO_A\singledbrc$ in terms of a reflexive coequalizer of coproducts of coinvariants over finite connected groupoids of finite tensor products), and the formula \eqref{o-comp-a-d} (of each color of $\sO \comp A$ as a coproduct of coinvariants over finite connected groupoids of finite tensor products).
\end{proof}

The next observation is the colored analogue of \cite{harper-gnt} (5.36) that we will need to use later.

\begin{proposition}
\label{o-ainfinity}
Suppose $\sO$ is a $\fC$-colored operad, $A \in \alg(\sO)$, $i : X \to Y \in \calmc$ is concentrated in one color $b \in \fC$, and
\begin{equation}
\label{o-ainfinity-pushout}
\nicexy{
\sO \circ X \ar[d]_-{\id \circ i} \ar[r]^-{f} 
& A \ar[d]^-{j}
\\
\sO \circ Y \ar[r] 
& A_{\infty}
}
\end{equation}
is a pushout in $\alg(\sO)$.  Suppose $d \in \fC$ and $[\uc] \in \sigmaofc$.  Then the object
\[
\sO_{A_{\infty}} \singledbrc \in \calm^{\sigmabrcop \times \{d\}}
\]
is isomorphic to a countable sequential colimit
\begin{equation}
\label{o-ainfinity-colim}
\colim\left(
\nicexy{
\osubazero \singledbrc \ar[r]^-{j_1}
& \osubaone \singledbrc \ar[r]^-{j_2}
& \osubatwo \singledbrc \ar[r]^-{j_3}
& \cdots}
\right),
\end{equation}
in which:
\begin{itemize}
\item
$\osubazero \singledbrc = \osuba \singledbrc \in \calm^{\sigmabrcop \times \{d\}}$;
\item
$j_t$ for $t \geq 1$ are defined inductively as pushouts
\begin{equation}
\label{osubjt}
\nicexy{
\osuba \singledbrtbbrc \tensorover{\Sigma_t} Q^{t}_{t-1} 
\ar[d]_-{\id \tensorover{\Sigma_t} i^{\boxprod t}} \ar[r]^-{f_*}
& 
\osubatminusone\singledbrc \ar[d]^-{j_t}
\\
\osuba \singledbrtbbrc \tensorover{\Sigma_t} Y^{\otimes t}
\ar[r]^-{\xi_t}
&
\osubat \singledbrc
}
\end{equation}
in $ \calm^{\sigmabrcop \times \{d\}}$, where $tb = (b,\ldots,b)$ with $t$ copies of $b$.
\end{itemize}
\end{proposition}

\begin{remark}
It is tempting to use the filtration \eqref{aoycolim} for $A_\infty$ and Lemma \ref{osubdc} to conclude that $\osubainfinity$ is the sequential colimit of the $\sO_{A_t}$.  However, the filtration \eqref{aoycolim} cannot be used this way here because it happens in $\calmc$, not in $\alg(\sO)$.
\end{remark}

\begin{proof}[Proof of Proposition \ref{o-ainfinity}]
The pushout $A_\infty$ is also the reflexive coequalizer
\medskip
\[
\colim\left(
\nicexy{
A \coprod (\sO \comp X) \coprod (\sO \comp Y)
\ar@<3pt>[r]^-{f_*}
\ar@<-3pt>[r]_-{i_*} 
& 
A \coprod (\sO \circ Y) \ar@/_2pc/[l]
}\right)
\]
in $\alg(\sO)$.  By Lemma \ref{osubdc} $\osubainfinity \singledbrc$ is the reflexive coequalizer of the diagram
\medskip
\begin{equation}
\label{ainfinity-coequalizer}
\nicexy{
\sO_{A \coprod (\sO \comp X) \coprod (\sO \comp Y)} \singledbrc
\ar@<3pt>[r]^-{f_*}
\ar@<-3pt>[r]_-{i_*} 
&
\sO_{A \coprod (\sO \circ Y)} \singledbrc \ar@/_2pc/[l]
}
\end{equation}
in $\calm^{\sigmabrcop \times \{d\}}$.  Recall the decomposition \eqref{o-sub-a-plus-oy} (and Remark \ref{oaplusoy-one-colored-y}) for $\sO_{A \coprod (\sO \circ Y)} \singledbrc$:
\[
\sO_{A \coprod (\sO \comp Y)} \singledbrc 
\cong
\coprod_{t \geq 0}
\left[
\sO_A \singledbrtbbrc \tensorover{\Sigma_t} Y^{\otimes t}
\right]
\]
This decomposition also applies to
\[
\sO_{A \coprod (\sO \comp X) \coprod (\sO \comp Y)} \singledbrc
\cong
\sO_{A \coprod \sO \comp (X \coprod Y)} \singledbrc.
\]
Therefore, the reflexive coequalizer $Z$ of \eqref{ainfinity-coequalizer} is characterized by the following universal properties:
\begin{enumerate}
\item
For each $t \geq 0$ there is a map
\[
\nicexy{
\sO_A \singledbrtbbrc \tensorover{\Sigma_t} Y^{\otimes t}
\ar[r]^-{\phi_t}
& Z \in \calm^{\sigmabrcop \times \{d\}},
}\]
where $tb = (b,\ldots,b)$ with $t$ copies of $b$.
\item
For any $s,t \geq 0$ the diagram
\[
\nicexy{
\osuba \smallbinom{d}{[sb];[tb]; [\uc]}
\tensorover{\Sigma_s \times \Sigma_t}
\left[X^{\otimes s} \otimes Y^{\otimes t}\right]
\ar[d]_-{i_*} \ar[r]^-{f_*}
&
\sO_A \singledbrtbbrc \tensorover{\Sigma_t} Y^{\otimes t}
\ar[d]^-{\phi_t}
\\
\osuba \smallbinom{d}{[(s+t)b]; [\uc]} \tensorover{\Sigma_{s+t}} Y^{\otimes (s+t)}
\ar[r]^-{\phi_{s+t}}
& 
Z
}\]
in $\calm^{\sigmabrcop \times \{d\}}$ is commutative, where $\osuba \smallbinom{d}{[sb];[tb]; [\uc]}$ was defined in \eqref{xdabc} for an arbitrary $\fC$-colored symmetric sequence.
\item
$Z$ is initial with respect to the above two properties.
\end{enumerate}
The rest of the proof is about checking that the sequential colimit \eqref{o-ainfinity-colim} has the above universal properties of $Z$.  This argument is very similar to the proof of Proposition \ref{colored_jpaa7.12}, so we will omit the details.
\end{proof}

\subsection{Homotopical Analysis of Pushouts}

Now we assume further that $\calm$ is a monoidal model category in the sense of \cite{ss} (3.1).  This subsumes the assumption that $\calm$ is symmetric monoidal closed with all small limits and colimits.  The extra assumption is that $\calm$ is a model category satisfying the pushout product axiom.  In particular, we are \emph{not} assuming the unit axiom, which is fine as long as we work at the model category level rather than on the level of homotopy categories.

We will need the following fact about diagram categories indexed by groupoids.  It is the groupoid version of \cite{bm06} (2.5.1, second part).

\begin{lemma}
\label{cof-restrict}
Suppose $\sg$ is a non-empty connected small groupoid, and $\iota : \sh \subseteq \sg$ is a non-empty connected sub-groupoid.  Then the restriction functor $\iota^* : \calmg \to \calmh$ takes (trivial) cofibrations to (trivial) cofibrations.  In particular, $\iota^*$ takes cofibrant objects to cofibrant objects.
\end{lemma}

\begin{proof}
We prove the assertion for cofibrations; the assertion for trivial cofibrations is proved similarly.

The diagram categories $\calmg$ and $\calmh$ are cofibrantly generated \cite{hirschhorn} (11.6.1).  Since cofibrations are closed under retracts and transfinite compositions \cite{hirschhorn} (10.3.4), it suffices to show that $\iota^*$ takes \emph{generating} cofibrations in $\calmg$ to cofibrations in $\calmh$.  A generating cofibration in $\calmg$ is a map of the form
\[
\varphi_g = \left(
\nicexy{
\coprod_{\sg(g;-)} X \ar[r]^-{\coprod i}
& \coprod_{\sg(g;-)} Y}\right)
\]
with $g \in \Ob(\sg)$ and $i : X \to Y \in \calm$ a generating cofibration.  Since $\sg$ is connected, for any two objects $g,g' \in \Ob(\sg)$, the maps $\varphi_g$ and $\varphi_{g'}$ are isomorphic.  So it suffices to show that $\iota^*$ takes \emph{one} $\varphi_g$ to a cofibration in $\calmh$.  Pick an object $h \in \Ob(\sh)$.  We will show that $\iota^*\varphi_h$ is a cofibration in $\calmh$.

The restriction $\iota^* \varphi_h \in \calmh$ has the same form as $\varphi_h$, but it only applies to objects in $\sh$.  Note that $\sh(h;h)$ is a group.  For any object $k \in \Ob(\sh)$, there is an $\sh(h;h)$-action
\[
\nicexy{\sh(h;h) \times \sg(h;k) \ar[r] 
& \sg(h;k)
}\]
on the set $\sg(h;k)$ induced by composition in $\sg$.  As is true for any group action on a set, there are natural isomorphisms
\[
\sg(h;k) 
\cong \coprod_{\mathrm{orbits}} \sh(h;h)
\cong \coprod_{\mathrm{orbits}} \sh(h;k)
\]
of $\sh(h;h)$-sets, where the coproducts are indexed by the set of $\sh(h;h)$-orbits in $\sg(h;k)$.  The isomorphism   
\[
\sh(h;h) \cong \sh(h;k)
\]
of $\sh(h;h)$-sets follows from the assumption that $\sh$ is connected.  Indeed, since $\sh$ is connected, we may pick an isomorphism $f : h \to k \in \sh$.  Then the above isomorphism is given by $g \longmapsto fg$ for $g \in \sh(h;h)$.  Going in the other direction, the isomorphism is given by $g \longmapsto f^{-1}g$ for $g \in \sh(h;k)$.

The cardinality of the set $\sg(h;k)/\sh(h;h)$ of orbits is independent of the object $k \in \Ob(\sh)$ because $\sg$ is connected.  In particular, it has the same cardinality as the set $\sg(h;h)/\sh(h;h)$ of orbits.  It follows that there is an isomorphism
\[
\iota^*\varphi_h 
\cong \coprod_{\sg(h;h)/\sh(h;h)} 
\underbrace{\left[ \nicexy{
\coprod_{\sh(h;-)} X \ar[r]^-{\coprod i}
& \coprod_{\sh(h;-)} Y} \right]}_{\phi_h}
\]
in $\calmh$.  The map $\phi_h$ is a generating cofibration in $\calmh$, so this coproduct is a cofibration in $\calmh$.
\end{proof}

The following observation, which we will use later, is inspired by \cite{harper-gnt} (5.44).  It says that $\sO_{(-)}$ has nice cofibrancy properties.

\begin{lemma}
\label{5.44-1}
Suppose $\sO$ is a $\fC$-colored operad, $A \in \alg(\sO)$, $i : X \to Y \in \calmc$, and
\[
\nicexy{
\sO \circ X \ar[d]_-{\id \circ i} \ar[r]^-{f} 
& A \ar[d]^-{j}
\\
\sO \circ Y \ar[r] 
& A_{\infty}
}\]
is a pushout in $\alg(\sO)$.   Suppose :
\begin{itemize}
\item
$i$ is a (trivial) cofibration in $\calmc$.
\item
The object $\sO_A \in \symseqcm$ is cofibrant, i.e., for all $d \in \fC$ and $[\uc] \in \sigmaofc$, the component component $\sO_A\singledbrc \in \calm^{\sigmabrcop \times \{d\}}$ is cofibrant.
\end{itemize}
Then the map $j_* : \sO_A \to \osubainfinity \in \symseqcm$ is a (trivial) cofibration, i.e., the map
\[
\nicexy{
\sO_A \singledbrc \ar[r]^-{j_*}
& \sO_{A_{\infty}} \singledbrc 
\in \calm^{\sigmabrcop \times \{d\}}
}\]
is a (trivial) cofibration for all $d \in \fC$ and $[\uc] \in \sigmaofc$.  In particular, $\osubainfinity \in \symseqcm$ is cofibrant.
\end{lemma}

\begin{proof}
Suppose $i$ is a cofibration in $\calmc$; the case when it is a trivial cofibration is proved similarly.

First observe that we may reduce to the case where $i$ is concentrated in a single color, say $b \in \fC$.  Indeed, $\calmc$ is a cofibrantly generated model category, in which each generating cofibration is concentrated in one color \cite{hirschhorn} (11.1.10).  So the cofibration $i$ is a retract of an $\sI$-cell complex, where $\sI$ is the set of generating cofibrations in $\calmc$.  A retract and transfinite induction argument implies that, if the assertion is true for $i \in \sI$, then it is true for all cofibrations in $\calmc$.  Therefore, we may assume that $i$ is concentrated in one color $b \in \fC$ such that  the $b$-colored entry of $i$ is a cofibration in $\calm$.

Since $i$ is concentrated in one color, we may use the filtration \eqref{o-ainfinity-colim} of $j_*$.   Since cofibrations are closed under pushouts and transfinite compositions \cite{hirschhorn} (10.3.4), to show that $j_*$ is a cofibration, it is enough to show that the left vertical map $\id \otimes_{\Sigma_t} i^{\boxprod t}$ in \eqref{osubjt} is a cofibration in $\calm^{\sigmabrcop \times \{d\}}$.  

Suppose $p : C \to D \in \calm^{\sigmabrcop  \times \{d\}}$ is a trivial fibration, i.e., an entry-wise trivial fibration in $\calm$ \cite{hirschhorn} (11.6.1).  Then the lifting problem
\[
\nicexy@R+10pt{
\osuba \singledbrtbbrc \tensorover{\Sigma_t} Q^{t}_{t-1} 
\ar[d]_-{\id \tensorover{\Sigma_t} i^{\boxprod t}} \ar[r]
& 
C \ar[d]^-{p}
\\
\osuba \singledbrtbbrc \tensorover{\Sigma_t} Y^{\otimes t}
\ar[r]_-{\xi_t} \ar@{.>}[ur]
&
D.
}\]
in $\calm^{\sigmabrcop  \times \{d\}}$ admits a dotted filler if and only if the adjoint lifting problem
\[
\nicexy@R+10pt{
\varnothing \ar[d] \ar[r]
& 
\Hom(Y^{\otimes t}, C) \ar[d]^-{(i^{\boxprod t},p)}
\\
\osuba \singledbrtbbrc \ar[r] \ar@{.>}[ur]
& 
\Hom(Q^t_{t-1}, C) \timesover{\Hom(Q^t_{t-1}, D)} \Hom(Y^{\otimes t}, D)
}\]
in $\calm^{\sigmaop_t \times \sigmabrcop \times \{d\}}$ admits a dotted lift.  Since the object
\[
\sO_A\singledbrtbc \in \calm^{\sigmabrtbcop \times \{d\}}
\]
is cofibrant by assumption, its restriction (Remark \ref{rk:xdac-restriction})
\[
\sO_A  \singledbrtbbrc \in \calm^{\sigmaop_t \times \sigmabrcop \times \{d\}}
\]
is also cofibrant by Lemma \ref{cof-restrict}.  Therefore, it suffices to show that the right vertical map $(i^{\boxprod t},p)$ is a trivial fibration in $\calm^{\sigmaop_t \times \sigmabrcop \times \{d\}}$, i.e., an entry-wise trivial fibration in $\calm$.  The iterated pushout product $i^{\boxprod t}$ is a cofibration in $\calm$ by the pushout product axiom.  Moreover, $p$ is an entry-wise trivial fibration in $\calm$.  So the pullback corner form of the pushout product axiom \cite{hovey} (4.2.2) implies that $(i^{\boxprod t}, p)$ is an entry-wise trivial fibration.
\end{proof}

Denote by $\calmc_{\cof}$ the collection of cofibrations in $\calmc = \prod_{\fC} \calm$.  Recall the adjunction \eqref{free-algebra-adjoint}.  The following observation is needed later when we apply the semi-model structure existence theorem \cite{fresse} (\ref{thm:fresse-semi-existence}).

\begin{proposition}
\label{colored-5.44}
Suppose $\sO$ is a $\fC$-colored operad, $A \in \alg(\sO)$, $i : X \to Y \in \calmc$, and
\begin{equation}
\label{j-pushoutof-i}
\nicexy{
\sO \circ X \ar[d]_-{\id \circ i} \ar[r]^-{f} 
& A \ar[d]^-{j}
\\
\sO \circ Y \ar[r] 
& A_{\infty}
}
\end{equation}
is a pushout in $\alg(\sO)$.  Suppose:
\begin{itemize}
\item
$i \in \calmc$ is a (trivial) cofibration.
\item
$A \in \alg(\sO)$ is an $\left(\sO \comp \calmc_{\cof}\right)$-cell complex. 
\item
 $\sO \in \symseqcm$ is cofibrant.
\end{itemize}
Then the underlying map of $j \in \calmc$ is also a (trivial) cofibration.
\end{proposition}

\begin{proof}
Suppose $i$ is a cofibration; the case when it is a trivial cofibration is proved similarly.  

Write $\sI$ (resp., $\sJ$) for the set of generating cofibrations (resp., generating trivial cofibrations) in $\calmc$.  Each map in $\sI \coprod \sJ$ is concentrated in one color \cite{hirschhorn} (11.1.10).  Since $\calmc$ is a  cofibrantly generated model category with generating cofibrations $\sI$, the map $i$ is a retract of a relative $\sI$-cell complex.  The functor $\sO \comp - : \calmc \to \alg(\sO)$ commutes with colimits (in particular, filtered colimits) because it is a left adjoint.  Therefore, it is enough to consider the case where $i \in \sI$, or more generally a cofibration in $\calmc$ concentrated in one color $c \in \fC$.

We now use the filtration \eqref{aoycolim} for the underlying map of $j \in \calmc$.  Since cofibrations are closed under transfinite compositions \cite{hirschhorn} (10.3.4), it suffices to show that each $j_t$ for $t \geq 1$ is an entry-wise cofibration.  Since $i \in \calmc$ is concentrated in one color $c \in \fC$, for each color $d \in \fC$, the $d$-colored entry of $j_t$ is given by the pushout \eqref{one-colored-jt-pushout}.  So it is enough to show that the left vertical map $\id \otimes_{\Sigma_t} i^{\boxprod t}$ there is a cofibration, where the identity map is for $\sO_A\singledbrtc$.

Note that taking coinvariants $(-)_{\sigmaop_t} : \calm^{\sigmaop_t} \to \calm$ is a left Quillen functor, the right adjoint being the constant diagram functor.  Since 
\[
\sO_A\singledbrtc \tensorover{\Sigma_t} i^{\boxprod t}
= \left[\left(\varnothing \to \sO_A\singledbrtc\right) \boxprod i^{\boxprod t}\right]_{\sigmaop_t},
\]
it is enough to show that the pushout product
\[
\left(\varnothing \to \sO_A\singledbrtc\right) \boxprod i^{\boxprod t} \in \calm^{\sigmaop_t}
\]
is a cofibration.  The iterated pushout product $i^{\boxprod t} \in \calm^{\sigmaop_t}$ is an underlying cofibration in $\calm$ by the pushout product axiom.  Therefore, by \cite{bm06} (2.5.2) it suffices to show that $\sO_A \singledbrtc \in \calm^{\sigmaop_t}$ is cofibrant.  In particular, it is enough to show that $\sO_A \in \symseqcm$ is cofibrant.

By the cofibrancy assumption on $\sO \in \symseqcm$ and the isomorphism \eqref{oemptyiso}, it is enough to show that the map
\[
\nicexy{
\sO \cong \sO_{\varnothing} \ar[r]
& \sO_A \in \symseqcm
}\]
induced by $\varnothing \to A \in \alg(\sO)$ is a cofibration.  By assumption the map $\varnothing \to A$ is a transfinite composition of pushouts of maps in $\sO \comp \calmc_{\cof}$.  So a transfinite induction using Lemma \ref{5.44-1} repeatedly proves that $\sO_{\varnothing} \to \sO_A \in \symseqcm$ is a cofibration.
\end{proof}

\section{Model Structures on Algebras over Colored Operads}
\label{sec:model-on-algebras}

In this section we will find conditions on a monoidal model category $\M$ and/or a colored operad $\sO$ so that $\sO$-algebras inherit a (semi-)model structure from $\M$. For a monad $T$, the category of $T$-algebras is said to \textit{inherit} a model structure from $\M$ if the weak equivalences (resp. fibrations) of $T$-algebras are maps that are weak equivalences (resp. fibrations) in $\M$. We refer to this as the \textit{projective} (semi-)model structure.

In each of the following three subsections we make use of the filtration of the preceding sections.  In \ref{subsec:all-operads-admissible} we extend a result from \cite{harper-jpaa} to the colored setting and prove that if $\M$ satisfies strong cofibrancy hypotheses (e.g. if $\M$ is chain complexes over a field of characteristic zero) then all operads are admissible, i.e., the category of algebras inherits a projective model structure.  In \ref{subsec:entrywise-cofibrant-admissible} we extend a result from \cite{white-thesis} to the case of colored operads, and prove that one can distribute this ``cofibrancy price'' between the operad and the model category so that the category of algebras over an  entrywise cofibrant colored operad inherits a projective semi-model structure with minimal hypotheses on $\M$.  Lastly,  in \ref{subsec:sigma-cof-admissible} we recover the fact (proven in the appendix of \cite{gutierrez-rondigs-spitzweck-ostvaer}) that algebras over colored operads which are cofibrant in $\symseqcm$ inherit projective semi-model structures. This is to say, for sufficiently cofibrant colored operads, almost no hypotheses are needed on $\M$ in order to have a good homotopy theory of operad-algebras.  In all three settings we include results proving that cofibrations of algebras with cofibrant source forget to cofibrations in $\M$.

\subsection{All Colored Operads}
\label{subsec:all-operads-admissible}

The following result says that, under a suitable cofibrancy assumption on $\calm$, every colored operad is admissible. For ease of exposition we have chosen to assume that the domains of the generating (trivial) cofibrations of our model category are small (such model categories are called \textit{strongly cofibrantly generated} in \cite{jy1}). In fact, the following results could be proven with lesser smallness hypotheses, though the statements would be more technical. We leave this extension to the interested reader.

\begin{theorem}
\label{colored-model}
Suppose $\calm$ is a strongly cofibrantly generated monoidal model category.  Suppose that 
\begin{quote}
$(\spadesuit)$ : for each $n \geq 1$ and for each object $X \in \calm^{\sigmaop_n}$, the function
\[
X \tensorover{\Sigma_n} (-)^{\boxprod n} : \calm \to \calm
\]
preserves trivial cofibrations.  
\end{quote}
Then for each $\fC$-colored operad $\sO$, the category $\alg(\sO)$ admits a projective model structure with weak equivalences and fibrations created in $\calmc$.  Moreover, this model structure is cofibrantly generated in which the set of generating (trivial) cofibrations is $\sO \comp \sI$ (resp., $\sO \comp \sJ$), where $\sI$ (resp., $\sJ$) is the set of generating (trivial) cofibrations in $\calmc$.
\end{theorem}

\begin{proof}
We will use Kan's Lifting Theorem \cite{hirschhorn} (11.3.2) on the adjunction
\[
\nicexy{
\calmc \ar@<2pt>[r]^-{\sO \comp -} 
& \alg(\sO) \ar@<2pt>[l]
}\]
in \eqref{free-algebra-adjoint}.  The Cartesian product $\calmc$ is also strongly cofibrantly generated by \cite{hirschhorn} (11.1.10), in which each generating (trivial) cofibration is concentrated in one color and is a generating (trivial) cofibration of $\calm$ there.  Let us now check the conditions in \cite{hirschhorn} (11.3.2).
\begin{enumerate}
\item
The category $\alg(\sO)$ has all small limits and colimits by Proposition \ref{algebra-bicomplete}.
\item
Since the forgetful functor $\alg(\sO) \to \calmc$ preserves filtered colimits (by Proposition \ref{algebra-bicomplete}) and since the domains in $\sI$ and $\sJ$ are small ($=$ the \emph{strongly} assumption), the domains of $\sO \comp \sI$ and $\sO \comp \sJ$ are also small.  So $\sO \comp \sI$ and $\sO \comp \sJ$ permit the small object argument.  This checks \cite{hirschhorn} 11.3.2(1).
\item
Finally, we need to check that every relative $(\sO \comp \sJ)$-cell complex is an underlying weak equivalence in $\calmc$ (i.e., an entry-wise weak equivalence).  We will prove slightly more.  We claim that every relative $(\sO \comp \sJ)$-cell complex is an underlying trivial cofibration in $\calmc$.  Since the model structure on $\calmc$ is defined entry-wise \cite{hovey} (1.1.6) and since trivial cofibrations are closed under transfinite compositions \cite{hirschhorn} (10.3.4), it is enough to consider a single pushout
\[
\nicexy{
\sO \circ X \ar[d]_-{\id \circ i} \ar[r]^-{f} 
& A_0 \ar[d]^-{j}
\\
\sO \circ Y \ar[r] 
& A_{\infty}
}\]
in $\alg(\sO)$, in which $i : X \to Y \in \sJ$.  In particular, $i$ is concentrated in a single color and is a generating trivial cofibration of $\calm$ there.  We must show that  $j$ is entry-wise trivial cofibration in $\calm$.  

By the filtration \eqref{aoycolim}, it suffices to show that each map
\[
j_t : (A_{t-1})_d \to (A_t)_d
\]
is a trivial cofibration in $\calm$ for each color $d \in \fC$ and $t \geq 1$.  Since $i \in \sJ$ is concentrated in a single color, by the pushout \eqref{one-colored-jt-pushout}, it is enough to show that $\id \tensorover{\Sigma_t} i^{\boxprod t}$ is a trivial cofibration in $\calm$.  By our hypothesis $(\spadesuit)$, it is now enough to observe that $i$ is a trivial cofibration in $\calm$.
\end{enumerate}
All the conditions in \cite{hirschhorn} (11.3.2) have now been checked.
\end{proof}

\begin{remark}
In the special case of $1$-colored operads (i.e., when $\fC = \{*\}$), Theorem \ref{colored-model} is a slight improvement of (the algebra part of) \cite{harper-jpaa} (Theorem 1.4), which assumes that every symmetric sequence is cofibrant.  Indeed, when every symmetric sequence is cofibrant, the condition $(\spadesuit)$ follows from the pushout product axiom.
\end{remark}

\begin{remark} \label{remark-more-gen-spadesuit}
In the one-colored case, Theorem \ref{colored-model} first appeared in \cite{white-thesis}. A result similar to \ref{colored-model} (which also holds for colored operads) has recently appeared in the preprint \cite{dmitri}. These two results have different hypotheses, but both extend \cite{harper-jpaa} to the setting of colored operads.

Our proof allows for more general results in \ref{subsec:entrywise-cofibrant-admissible} and \ref{subsec:sigma-cof-admissible} that remove or weaken the hypothesis $(\spadesuit)$ (and so hold for a wider class of model categories) and that result in semi-model structures on categories of algebras. 

Observe that $(\spadesuit)$ could be weakened to state that $X \otimes_{\Sigma_n} (-)^{\boxprod n}$ takes trivial cofibrations into some class of morphisms contained in the weak equivalences and closed under transfinite composition and pushout (e.g. the trivial $h$-cofibrations of \cite{batanin-berger}) as our filtration shows $j_t$ and $j$ would be in this class, hence weak equivalences.
\end{remark}

\begin{example}
Theorem \ref{colored-model} applies to all the colored operads in section \ref{subsec:def-colored-algebra}.  For instance, by the isomorphism \eqref{colored-op-algebras}, for each non-empty set of colors $\fC$, the category $\operadsigmac$ of $\fC$-colored operads in $\calm$ inherits a cofibrantly generated model structure, where weak equivalences and fibrations are created entry-wise in $\calm$.
\end{example}

The next observation says that every $\sO$-algebra cofibration with a cofibrant domain (resp., every cofibrant $\sO$-algebra) is an underlying cofibration (resp., cofibrant object), provided  that $\sO$ is $\Sigma$-cofibrant.

\begin{proposition}
\label{underlying-cof}
Suppose $\calm$ is a cofibrantly generated monoidal model category.  Suppose $\sO$ is a $\fC$-colored operad such that:
\begin{itemize}
\item
$\alg(\sO)$ inherits from $\calmc$ a cofibrantly generated projective model structure (which holds, e.g., in the context of Theorem \ref{colored-model});
\item
$\sO \in \symseqcm$ is cofibrant.
\end{itemize}
Then the following statements hold.
\begin{enumerate}
\item
If $j : A \to B \in \alg(\sO)$ is a cofibration with $A$ cofibrant in $\alg(\sO)$, then the underlying map $j \in \calmc$ is a cofibration.
\item
Every cofibrant $\sO$-algebra is cofibrant in $\calmc$.
\end{enumerate}
\end{proposition}

\begin{proof}
First consider part (1).  Write $\sI$ for the set of generating cofibrations in $\calmc$.  By assumption the cofibration $j$ is a retract of a relative $(\sO \comp \sI)$-cell complex.  By a retract and transfinite induction argument, it is enough to prove that, given the pushout \eqref{j-pushoutof-i} with $i \in \sI$ and $A \in \alg(\sO)$ cofibrant, the map $j \in \calmc$ is a cofibration.  Proposition \ref{colored-5.44} says that this assertion is true whenever $A$ is an $(\sO \comp \calmc_{\cof})$-cell complex, which includes any $(\sO \comp \sI)$-cell complex.  By a retract argument, this assertion is also true for retracts of $(\sO \comp \sI)$-cell complexes, i.e., cofibrant $\sO$-algebras, proving part (1).

Write $\emptyo$ for the initial $\sO$-algebra \eqref{initial-o-algebra}.  For part (2), suppose $B \in \alg(\sO)$ is cofibrant, so $\iota : \emptyo \to B$ is a cofibration in $\alg(\sO)$.  We want to show that the underlying object $B \in \calmc$ is cofibrant.  Since $\emptyo$, being the initial object, is a cofibrant $\sO$-algebra, part (1) implies that the underlying map $\iota \in \calmc$ is a cofibration.  So it suffices to show that $\emptyo \in \calmc$ is cofibrant, i.e., color-wise cofibrant in $\calm$.  For each color $d \in \fC$, the $d$-colored entry of $\emptyo$ is $\sO\singledempty$, which is the $\singledempty$-component of $\sO$.  But by assumption every $\singledbrc$-component of $\sO$ is cofibrant, which holds in particular when $\uc = \varnothing$.
\end{proof}

\begin{example}
In the context of Theorem  \ref{colored-model}, Proposition \ref{underlying-cof} applies to the colored operad $\opc$ in Example \ref{ex:operad-of-operad}, which is $\Sigma$-cofibrant.  In other words, in the model category $\operadsigmac \cong \alg(\opc)$, every cofibration with a cofibrant domain is an entry-wise cofibration in $\calm$, and every cofibrant $\fC$-colored operad is entry-wise cofibrant in $\calm$.
\end{example}

\subsection{Entrywise Cofibrant Colored Operads}
\label{subsec:entrywise-cofibrant-admissible}

The following condition should be compared to $(\spadesuit)$ in Theorem \ref{colored-model}.

\begin{definition}
Suppose $\calm$ is a symmetric monoidal category and is a model category.  Define the following condition.
\begin{quote}
$(\clubsuit)$: For each $n \geq 1$ and $X \in \calm^{\sigmaop_n}$ that is cofibrant in $\calm$, the function
\[
X \tensorover{\Sigma_n} (-)^{\boxprod n} : \calm \to \calm
\]
preserves (trivial) cofibrations.
\end{quote}
The condition $(\clubsuit)$ for cofibrations will be referred to as $\clubcof$, and the condition for trivial cofibrations as $\clubacof$.  So $(\clubsuit) = \clubcof + \clubacof$.
\end{definition}

Recall from \cite{white-commutative} and \cite{lurie} that \emph{Lurie's commutative monoid axiom} says that if $f \in \calm$ is a cofibration, then $f^{\boxprod n} \in \calm^{\Sigma_n}$ is a $\Sigma_n$-equivariant cofibration for each $n \geq 1$.

\begin{proposition}
In each monoidal model category, Lurie's commutative monoid axiom implies $\clubcof$. 
\end{proposition}

\begin{proof}
If $f \in \calm$ is a cofibration, then Lurie's commutative monoid axiom says that $f^{\boxprod n}$ is a cofibration in $\calm^{\sigmaop_n}$.  Suppose $X \in \calm^{\sigmaop_n}$ is underlying cofibrant in $\calm$.  Then
\[
X \tensorover{\Sigma_n} f^{\boxprod n} 
= 
\left((\varnothing \to X) \boxprod f^{\boxprod n}\right)_{\Sigma_n}
\]
is the image under the left Quillen functor $(-)_{\Sigma_n} : \calm^{\sigmaop_n} \to \calm$ of the pushout product of a map in $\calm^{\sigmaop_n}$ that is an underlying cofibration in $\calm$ with a $\Sigma_n$-equivariant cofibration $f^{\boxprod n}$.  By \cite{bm06} (2.5.1) such a pushout product is a $\Sigma_n$-equivariant cofibration, so its image under the left Quillen functor $(-)_{\Sigma_n}$ is a cofibration in $\calm$.
\end{proof}

\begin{theorem}
\label{middle-row-one-color}
Suppose 
$\calm$ is a cofibrantly generated monoidal model category satisfying $(\clubsuit)$.  Then for each entrywise cofibrant $\fC$-colored operad $\sO$ in $\calm$, the category $\alg(\sO)$ admits a cofibrantly generated \textbf{semi}-model structure over $\calm$  such that the weak equivalences and fibrations are created in $\calm$.  Moreover:
\begin{enumerate}
\item
If $j : A \to B \in \alg(\sO)$ is a cofibration with $A$ cofibrant in $\alg(\sO)$, then the underlying map of $j$ is entrywise a cofibration.
\item
Every cofibrant $\sO$-algebra is entrywise cofibrant in $\calm$.
\end{enumerate}
\end{theorem}

The proof of Theorem \ref{middle-row-one-color} uses the following observation.

\begin{lemma}
\label{middle-row-lemma}
Suppose $\calm$ is a symmetric monoidal closed category and is a model category satisfying $\clubcof$, and $\sO$ is a $\fC$-colored operad in $\calm$.
\begin{enumerate}
\item
Suppose $j : A \to B \in \alg(\sO)$ is a relative $(\sO \comp \calm_{\cof})$-cell complex, i.e., a retract of a transfinite composition of pushouts of maps in $\sO \comp \calm_{\cof}$.  Suppose also that $\sO_A$ is entrywise cofibrant in $\calm$.  Then $\sO_A \to \sO_B$ is entrywise a cofibration in $\calm$.
\item
Suppose $\sO$ is entrywise cofibrant in $\calm$, and suppose $A$ is an $(\sO \comp \calm_{\cof})$-cell complex, i.e., $\varnothing \to A \in \alg(\sO)$ is a relative $(\sO \comp \calm_{\cof})$-cell complex.  Then $\sO_A$ is entrywise cofibrant in $\calm$.
\end{enumerate}
\end{lemma}

\begin{proof}
Consider the first assertion.  Since entrywise cofibrations are closed under retracts and transfinite compositions, by a retract and transfinite induction argument, we may assume that $j$ is a pushout
\[
\nicexy{
\sO \comp X \ar[r]^-{f} \ar[d]_-{i_*}  
& A \ar[d]^-{j}\\
\sO \comp Y \ar[r] & B}
\]
in $\alg(\sO)$ for some cofibration $i : X \to Y \in \calm$.  Here we are regarding $i$ as a map of $\fC$-colored objects, both concentrated at the same single color.  We want to show that
\[
\sO_A\duc \to \sO_B\duc
\]
is a cofibration in $\calm$ for each $d \in \fC$ and each $\fC$-profile $\uc$.

By Proposition \ref{o-ainfinity} the map $\sO_A\duc \to \sO_B\duc$ is a countable composite of maps
\[
j_t : \sO_A^{t-1}\duc \to \sO_A^{t}\duc
\]
for $t \geq 1$, where $\sO_A^0 = \sO_A$.  Each map $j_t$ is a pushout in $\calm$ of a map of the form
\[
\sO_A(\vdots) \tensorover{\Sigma_t} i^{\boxprod t},
\]
where $\sO_A(\vdots)$ is an entry of $\sO_A$, and $i^{\boxprod t}$ is the $t$-fold iterated pushout product of $i : X \to Y$.

Since cofibrations are closed under countable compositions, it suffices to show that each $j_t$ is a cofibration in $\calm$.   As pushouts of cofibrations are cofibrations, this reduces to proving $\sO_A(\vdots) \otimes_{\Sigma_t}i^{\boxprod t}$ is a cofibration in $\calm$.  This last condition holds by $\clubcof$, the fact that $i$ is a cofibration in $\calm$, and the hypothesis that $\sO_A$ is entrywise cofibrant.  This proves the first assertion.

For the second assertion, apply (1) in the case $\varnothing \to A$, where $\varnothing$ is the initial $\sO$-algebra.  Part (1) is applicable because, by \eqref{o-sub-empty}, $\sO_\emptyset \cong \sO$ at each entry, and $\sO$ is entrywise cofibrant by assumption.  Part (1) guarantees that, at each $\duc$-entry,
\[
\sO\duc \cong \sO_{\emptyset}\duc \to \sO_A\duc
\]
is a cofibration in $\calm$.  This implies the cofibrancy of $\sO_A\duc$.
\end{proof}

The key step in the proof of Theorem \ref{middle-row-one-color} is the next observation.  It verifies condition $(*)$ in the semi-model structure existence theorem in \cite{fresse} (\ref{thm:fresse-semi-existence}).

\begin{proposition}
\label{1-colored-5.44}
Suppose $\calm$ is a monoidal model category satisfying $(\clubsuit)$, and $\sO$ is an entrywise cofibrant $\fC$-colored operad in $\calm$.  Suppose
\[
\nicexy{
\sO \comp X \ar[d]_-{\sO \comp i} \ar[r] 
& A \ar[d]^-{j}
\\
\sO \comp Y \ar[r] & A_{\infty}
}\]
is a pushout in $\alg(\sO)$ such that 
\begin{itemize}
\item
$i : X \to Y \in \calm$ is a (trivial) cofibration, regarded as a map concentrated at a single color, and 
\item
$A$ is an $(\sO \comp \calm_{\cof})$-cell complex.  
\end{itemize}
Then the underlying map of $j$ is entrywise a (trivial) cofibration.
\end{proposition}

\begin{proof}
Cofibrations (resp., trivial cofibrations) are closed under pushouts and  transfinite compositions.  So by the filtration \eqref{one-colored-jt-pushout} , it is enough to show that the map
\[
\nicexy{
\sO_A(\vdots) \tensorover{\Sigma_t} Q^t_{t-1}
\ar[r]^-{\id \tensorover{\Sigma_t} i^{\boxprod t}}
& 
\sO_A(\vdots) \tensorover{\Sigma_t} Y^{\otimes t}
}\]
is a (trivial) cofibration in $\calm$ for each $t \geq 1$.  By Lemma \ref{middle-row-lemma}(2), $\sO_A(\vdots)$ is cofibrant in $\calm$.  So $(\clubsuit) = \clubcof + \clubacof$ says that $\sO_A(\vdots) \otimes_{\Sigma_t} (-)^{\boxprod t}$ preserves (trivial) cofibrations.  Since $i$ is a (trivial) cofibration, this finishes the proof.
\end{proof}

\begin{proof}[Proof of Theorem \ref{middle-row-one-color}]
We will employ the semi-model structure existence result \cite{fresse} (\ref{thm:fresse-semi-existence}) to the free-forgetful adjunction
\[
\nicexy{
\calm \ar@<2pt>[r]^-{\sO \comp -} 
& \alg(\sO). \ar@<2pt>[l]
}\]
Let us now check the conditions needed to apply Theorem \ref{thm:fresse-semi-existence}.
\begin{enumerate}
\item
$\alg(\sO)$ has all small limits and colimits by \cite{harper-jpaa} (5.16 and 5.19).
\item
The forgetful functor $\alg(\sO) \to \calm$ preserves colimits over non-empty ordinals by \cite{harper-jpaa} (5.15).
\item
The condition $(*)$ in Theorem \ref{thm:fresse-semi-existence} holds by Proposition \ref{1-colored-5.44}.
\end{enumerate}
All the conditions in Theorem \ref{thm:fresse-semi-existence} have now been checked, so we have proved all the assertions except for the last one about cofibrant $\sO$-algebras being underlying cofibrant.

Write $\emptyo$ for the initial $\sO$-algebra.  For statement (2), suppose $B \in \alg(\sO)$ is cofibrant, so $\iota : \emptyo \to B$ is a cofibration in $\alg(\sO)$.  We want to show that the underlying object of $B$ is entrywise cofibrant.  Since $\emptyo$, being the initial object, is a cofibrant $\sO$-algebra, statement (1) implies that the underlying map of $\iota$ is entrywise a cofibration.  So it suffices to show that $\emptyo$ is entrywise cofibrant.  This is true because
\[
(\emptyo)_d = \sO\singledempty,
\]
which is cofibrant by assumption.
\end{proof}

\begin{remark}
The proof method of this theorem demonstrates that the precise hypothesis required in order to obtain a semi-model structure on $\alg(\sO)$ using the filtration of Section \ref{sec:alg-over-colored} is that for all $\sO$-algebras $A$ and all generating trivial cofibrations $i$ of $\M$, the maps $\sO_A\smallbinom{d}{[tc]} \otimes_{\Sigma_t} i^{\boxprod t}$ are all trivial cofibrations for all $t \geq 1$ and $c \in \fC$. In the case of $1$-colored operads, this hypothesis was introduced in \cite{white-thesis} as the \textit{$\sO$-algebra axiom}. In this paper we have preferred to work with $(\spadesuit)$ and $(\clubsuit)$ because they are easier to check in practice, due to the complexity of the formulas for $\sO_A$.
\end{remark}

\subsection{$\Sigma$-Cofibrant Colored Operads}
\label{subsec:sigma-cof-admissible}

The next result says that, if we do not impose any extra cofibrancy condition, such as  $(\spadesuit)$ or $(\clubsuit)$, on $\calm$, then $\Sigma$-cofibrant colored operads are \emph{semi-admissible} in the sense that the category of algebras admits a suitable semi-model structure.

\begin{theorem}
\label{sigmacof-alg-semi}
Suppose $\calm$ is a cofibrantly generated monoidal model category.  Suppose $\sO$ is a $\fC$-colored operad such that $\sO \in \symseqcm$ is cofibrant.  Then $\alg(\sO)$ is a cofibrantly generated semi-model category over $\calmc$ such that the weak equivalences and fibrations are created in $\calmc$.  Moreover:
\begin{enumerate}
\item
If $j : A \to B \in \alg(\sO)$ is a cofibration with $A$ cofibrant in $\alg(\sO)$, then the underlying map $j \in \calmc$ is a cofibration.
\item
Every cofibrant $\sO$-algebra is cofibrant in $\calmc$.
\end{enumerate}
\end{theorem}

\begin{proof}
We will employ the semi-model structure existence result \cite{fresse} (\ref{thm:fresse-semi-existence}) to the adjunction
\[
\nicexy{
\calmc \ar@<2pt>[r]^-{\sO \comp -} 
& \alg(\sO) \ar@<2pt>[l]
}\]
in \eqref{free-algebra-adjoint}, where $\calmc$ is a cofibrantly generated model category \cite{hirschhorn} (11.1.10).  Let us now check the conditions needed to apply Theorem \ref{thm:fresse-semi-existence}.
\begin{enumerate}
\item
$\alg(\sO)$ has all small limits and colimits by Proposition \ref{algebra-bicomplete}.
\item
The forgetful functor $\alg(\sO) \to \calmc$ preserves colimits over non-empty ordinals also by Proposition \ref{algebra-bicomplete}.
\item
The condition $(*)$ in Theorem \ref{thm:fresse-semi-existence} holds by Proposition \ref{colored-5.44}.
\end{enumerate}
All the conditions in Theorem \ref{thm:fresse-semi-existence} have now been checked, so we have proved all the assertions except for the last one about cofibrant $\sO$-algebras being underlying cofibrant.  This assertion is proved by exactly the same argument as in Proposition \ref{underlying-cof} (2).
\end{proof}

\section{Preservation of Algebras Under Bousfield Localization}
\label{sec:preservation}

In this section we provide an application of the (semi-)model structures of the previous section to the problem of preservation of algebraic structure under Bousfield localization.  We will prove a general preservation result, Theorem \ref{thm:general-preservation}, for algebras over a monad under Bousfield localization.  Then we will use this preservation result and the (semi-)model category existence results from the previous section to obtain preservation result for algebras over a colored operad under Bousfield localization.

\subsection{Bousfield Localization}
\label{subsec:bous-loc}

Let us first remind the reader about the process of Bousfield localization as discussed in \cite{hirschhorn}. This is a general machine that
starts with a (nice) model category $\calm$ and a set of morphisms $\C$ and produces a new model structure $L_\C(\calm)$ on the same category in which maps in $\C$ are now weak equivalences.  Roughly speaking, in going from $\M$ to $L_\C(\M)$, we keep the cofibrations the same and add more weak equivalences.

Furthermore, this is done in a universal way, introducing the smallest number of new weak equivalences possible. When we say Bousfield localization we will always mean \emph{left} Bousfield localization. So the cofibrations in $L_\C(\calm)$ will be the same as the cofibrations in $\calm$.

Bousfield localization proceeds by first constructing the fibrant objects of $L_\C(\calm)$ and then constructing the weak equivalences. In both cases this is done via simplicial mapping spaces $\map(-,-)$. If $\calm$ is a simplicial or topological model category then one can use the hom-object in $sSet$ or $Top$. Otherwise a framing is required to construct the simplicial mapping space. We refer the reader to \cite{hovey} or \cite{hirschhorn} for details on this process.

An object $N$ is said to be \textit{$\C$-local} if it is fibrant in $\calm$ and if for all $g:X\to Y$ in $\C$, the map 
\[
\nicexy@C+.5cm{\map(Y,N)\ar[r]^-{\map(g,N)} & \map(X,N)}
\]
is a weak equivalence in $sSet$. These objects are precisely the fibrant objects in $L_\C(\calm)$. A map $f:A\to B$ is a \textit{$\C$-local equivalence} if for all $\C$-local objects $N$, the map
\[
\nicexy@C+.5cm{\map(B,N)\ar[r]^-{\map(f,N)} & \map(A,N)}
\]
is a weak equivalence. These maps are precisely the weak equivalences in $L_\C(\calm)$. A $\C$-local equivalence between $\C$-local objects is a weak equivalence in $\M$.

\begin{remark}
\label{remark:loc-wrt-cofibrations}
Throughout this paper we assume $\C$ is a set of cofibrations between cofibrant objects. This can always be guaranteed in the following way. For any map $f$, let $Qf$ denote the cofibrant replacement, and let $\tilde{f}$ denote the left factor in the cofibration-trivial fibration factorization of $Qf$. Then $\tilde{f}$ is a cofibration between cofibrant objects and we may define $\tilde{\C} = \{\tilde{f}\;|\; f\in \C\}$. Localization with respect to $\tilde{\C}$ yields the same result as localization with respect to $\C$, so our assumption that the maps in $\C$ are cofibrations between cofibrant objects loses no generality.
\end{remark}

\emph{We also assume everywhere that the model category $L_\C(\calm)$ exists and is cofibrantly generated.} This can be guaranteed by assuming $\M$ is left proper and either combinatorial (as discussed in \cite{barwickSemi}) or cellular (as discussed in \cite{hirschhorn}). A model category is \textit{left proper} if pushouts of weak equivalences along cofibrations are again weak equivalences. We will make this a standing hypothesis on $\calm$. However, as we have not needed the cellularity or combinatoriality assumptions for our work we have decided not to assume them. In this way if a Bousfield localization is known to exist for some reason other than the theory in \cite{hirschhorn} then our results will be applicable.

On the model category level the functor $L_\C$ is the identity. So when we write $L_\C$ as a functor we shall mean the composition of derived functors
\[
\nicexy{\Ho(\M)\ar[r] & \Ho(L_\C(\M)) \ar[r] & \Ho(\M)},
\]
i.e., $E\to L_\C(E)$ is the unit map of the adjunction $\Ho(\M) \rightleftarrows \Ho(L_\C(\M))$. In particular, for any $E$ in $\M$, $L_\C(E)$ is weakly equivalent to $R_\C Q E$, where $R_\C$ is a choice of fibrant replacement in $L_\C(\M)$ and $Q$ is a cofibrant replacement in $\M$.

\subsection{Preservation of  Monadic  Algebras}

Note that every monad on $\M$ is also a monad on $L_\C(\M)$ because  it has the same underlying category.  If $\M$ is a monoidal category, then so is $L_\C(\M)$, and a colored operad on $\M$ is also a colored operad on $L_\C(\M)$.

\begin{definition}
Assume that $\M$ and $L_\C(\M)$ are model categories, and $T$ is a monad on $\M$. Then $L_\C$ is said to \textit{preserve $T$-algebras} if the following two statements hold.

\begin{enumerate}
\item When $E$ is a $T$-algebra, there is some $T$-algebra $\tilde{E}$ which is weakly equivalent in $\M$ to $L_\C(E)$.

\item In addition, when $E$ is a cofibrant $T$-algebra, then there is a natural choice of $\tilde{E}$ and a lift of the localization map $E\to L_\C(E)$ to a $T$-algebra homomorphism $E\to \tilde{E}$.
\end{enumerate}
If $\M$ and $L_\C(\M)$ are monoidal model categories and if $\sO$ is a $\fC$-colored operad on $\M$, say that $L_\C$ \emph{preserves $\sO$-algebras} if it preserves $T$-algebras for the free $\sO$-algebra monad $T = \sO \circ -$.
\end{definition}

\begin{remark}
The notion of preservation was also considered in \cite{CGMV}, but only for cofibrant $E$. A more general notion of preservation was considered in \cite{casacuberta-raventos-tonks} in the setting of homotopical categories and general monads, but because of the generality the conditions for preservation presented in \cite{casacuberta-raventos-tonks} are different from ours. When \cite{casacuberta-raventos-tonks} specializes to operads, only simplicial operads are considered.
\end{remark}

We are now ready to state our general preservation result. This is a generalization of Theorem 3.2 and Corollary 3.4 of \cite{white-localization} from the setting of operads to the setting of monads.

\begin{theorem} \label{thm:general-preservation}
Let $\M$ be a model category such that the Bousfield localization $L_\C(\M)$ exists and is a model category. Let $T$ be an monad on $\M$. If the categories of $T$-algebras in $\M$ and in $L_\C(\M)$ inherit projective semi-model structures from $\M$ and $L_\C(\M)$, then $L_\C$ preserves $T$-algebras.
\end{theorem}

\begin{proof}
Let $\alg_{L_\C(\M)}(T)$ and $\alg_\M(T)$ denote the semi-model categories of $T$-algebras in $L_\C(\M)$ and $\M$, respectively. Let $R_\C$ denote fibrant replacement in $L_\C(\M)$, let $R_{\C,T}$ denote fibrant replacement in $\alg_{L_\C(\M)}(T)$, and let $Q_T$ denote cofibrant replacement in $\alg_{\M}(T)$. Due to the fact that $\alg_{L_\C(\M)}(T)$ is only a semi-model category, we shall only apply $R_{\C,T}$ to objects which are cofibrant in $\alg_{L_\C(\M)}(T)$. We will prove the first form of preservation and our method of proof will allow us to deduce the second form of preservation in the special case where $E$ is a cofibrant $T$-algebra.

In our proof, $\tilde{E}$ will be $R_{\C,T}Q_T(E)$. Because $Q$ is the left derived functor of the identity adjunction between $\M$ and $L_\C(\M)$, and $R_\C$ is the right derived functor of the identity, we know that $L_\C(E) \simeq R_\C Q(E)$ in $\M$. We must therefore show
\[
R_\C Q(E)\simeq R_{\C,T}Q_T(E)
\]
in $\M$.

The map $Q_T E\to E$ is a trivial fibration in $\alg_{\M}(T)$, hence in $\M$ by the definition of the projective semi-model structure. The map $QE\to E$ is also a weak equivalence in $\M$. Consider the following lifting diagram in $\M$:
\begin{align} \label{diagram-cof-rep-lift}
\nicexy{\emptyset \ar[r] \ar@{^{(}->}[d] & Q_T E \ar@{->>}[d]^\simeq \\
QE \ar[r] \ar@{..>}[ur] & E}
\end{align}
The lifting axiom gives the dotted map $QE\to Q_T E$, and it is necessarily a weak equivalence in $\M$ by the 2-out-of-3 property.

Since $Q_T E$ is a $T$-algebra in $\M$ it must also be a $T$-algebra in $L_\C(\M)$. We may therefore construct the following lift in $L_\C(\M)$:
\begin{align*}
\nicexy{Q_T E \ar@{^{(}->}[d] \ar[r] & R_{\C,T} Q_T E \ar@{->>}[d]\\
R_\C Q_T E \ar[r] \ar@{..>}[ur] & \ast}
\end{align*}
In this diagram the left vertical map is a weak equivalence in $L_\C(\M)$, and the top horizontal map is a weak equivalence in $\alg_{L_\C(\M)}(T)$. Because the model category $\alg_{L_\C(\M)}(T)$ inherits weak equivalences from $L_\C(\M)$, the top horizontal map is a weak equivalence in $L_\C(\M)$. Therefore, by the 2-out-of-3 property, the dotted lift is a weak equivalence in $L_\C(\M)$. We make use of this map as the horizontal map in the lower right corner of the diagram below.

The top horizontal map $QE \to Q_T E$ in the following diagram is the first map we constructed in \eqref{diagram-cof-rep-lift}, which was proven to be a weak equivalence in $\M$. The square in the diagram below is then obtained by applying $R_\C$ to that map. In particular, $R_\C QE \to R_\C Q_T E$ is a weak equivalence in $L_\C(\M)$:
\begin{align*}
\nicexy{QE \ar[r] \ar[d] & Q_T E\ar[d] & \\ R_\C QE \ar[r] & R_\C Q_T E \ar[r] & R_{\C,T}Q_T E}
\end{align*}
We have shown that both of the bottom horizontal maps are weak equivalences in $L_\C(\M)$. Thus, by the 2-out-of-3 property, their composite $R_\C QE \to R_{\C,T}Q_TE$ is a weak equivalence in $L_\C(\M)$. All the objects in the bottom row are fibrant in $L_\C(\M)$, so these $\C$-local equivalences are actually weak equivalences in $\M$.

As $E$ was a $T$-algebra and $Q_T$ and $R_{\C,T}$ are endofunctors on categories of $T$-algebras, it is clear that $R_{\C,T}Q_TE$ is a $T$-algebra. We have just shown that $L_\C(E)$ is weakly equivalent to this $T$-algebra in $\M$, so we are done.

We turn now to the case where $E$ is assumed to be a cofibrant $T$-algebra. We have seen that there is an $\M$-weak equivalence $R_\C QE \to R_{\C,T}Q_T E$, and above we took $R_{\C,T}Q_T E$ in $\M$ as our representative for $L_\C(E)$ in $\Ho(\M)$. Because $E$ is a cofibrant $T$-algebra, the fibrant replacement $R_{\C,T}E$ exists in $\alg_{L_\C(\M)}(T)$. Furthermore, there are weak equivalences $E \leftrightarrows Q_T(E)$ in $\alg_{L_\C(\M)}(T)$ because all cofibrant replacements of a given object are weakly equivalent, e.g. by diagram (\ref{diagram-cof-rep-lift}). So passage to $Q_T(E)$ is unnecessary when $E$ is cofibrant, and we take $R_{\C,T} E$ as our representative for $L_\C(E)$. We may then lift the localization map $E\to L_\C(E)$ in $\Ho(\M)$ to the fibrant replacement map $E \to R_{\C,T} E$ in $\M$. As this fibrant replacement is taken in $\alg_{L_\C(\M)}(T)$, this map is a $T$-algebra homomorphism, as desired. Naturality follows from the functoriality of fibrant replacement.
\end{proof}

\subsection{Monoidal Bousfield Localization}

Shortly, we will apply this theorem to monads $T$ arising from the free $\sO$-algebra functor for $\fC$-colored operads $\sO$ in $\M$. However, to use the existence results from the previous section, we first need a way to check the hypothesis that $L_\C(\M)$ is a monoidal model category.  Recall that we always assume that $L_\C(\M)$ exists and is a cofibrantly generated model category.  So what we need is a checkable condition for when $L_\C(\M)$ satisfies the pushout product axiom.  Such a condition is given in the following definition taken from \cite{white-localization}.

\begin{definition}
A Bousfield localization $L_\C$ is said to be a \emph{monoidal Bousfield
localization} if $L_\C(\M)$ satisfies the pushout product axiom, the unit axiom (i.e., for any cofibrant $X$, the map $QI \otimes X \to I \otimes X \cong X$ is a weak
equivalence), and the axiom that cofibrant objects are flat.
\end{definition}

The following theorem, which will \emph{not} be used below, is proven in Theorem 4.6 in \cite{white-localization}.  It gives a checkable condition for a monoidal Bousfield localization, in particular the pushout product axiom in $L_\C(\M)$.  Although we will not use it in this paper, it is stated here to emphasize that being a monoidal Bousfield localization is indeed checkable in practice.

\begin{theorem} \label{thm:PPAxiom-nontractable}
Suppose $\M$ is a cofibrantly generated monoidal model category in which cofibrant objects are flat. Then $L_\C$ is a monoidal Bousfield localization if and only if every map of the form $f \otimes id_K$, where $f$ is in $\C$ and $K$ is cofibrant, is a $\C$-local equivalence.
\end{theorem}

If we know in addition that the domains of the generating cofibrations in $\M$ are cofibrant then Theorem 4.5 in \cite{white-localization} proves that it is sufficient to check the condition above for $K$ running through the domains and codomains of the generating cofibrations.

\subsection{Preservation of Operadic Algebras}

Now we consider preservation of algebras over a colored operad under Bousfield localization.  Putting together Theorems \ref{thm:general-preservation} and \ref{sigmacof-alg-semi}, we may deduce:

\begin{theorem}
\label{thm:pres-over-sigma-cofibrant}
Suppose that $\M$ is a cofibrantly generated monoidal model category, $L_\C(\M)$ is a monoidal model category, and $\sO$ is a $\fC$-colored operad in $\M$ that is cofibrant in $\symseqcm$. Then $L_\C$ preserves $\sO$-algebras.
\end{theorem}

\begin{proof}
We must verify the hypotheses of Theorem \ref{thm:general-preservation}. First, we apply Theorem \ref{sigmacof-alg-semi} to $\M$ to obtain a projective semi-model structure $\alg_\M(\sO)$. Next, we use our assumption that $L_\C(\M)$ is a monoidal model category (which can be verified via Theorem \ref{thm:PPAxiom-nontractable} for example), to apply Theorem \ref{sigmacof-alg-semi} to $L_\C(\M)$ and obtain a projective semi-model structure $\alg_{L_\C(\M)}(\sO)$.  Note that $\sO$ is also a $\fC$-colored operad in $L_\C(\M)$ that is cofibrant as a $\fC$-colored symmetric sequence in $L_\C(\M)$.  The reason is that, first of all, the generating cofibrations in $\calm^{\sigmabrcopd}$ are constructed from those in $\M$.  A similar statement holds with $L_\C(\M)$ in place of $\M$.  But the cofibrations in $\M$ and in $L_\C(\M)$ are the same.  So the entry of $\sO$ in $(L_\C(\M))^{\sigmabrcopd}$ is cofibrant.

Finally, Theorem \ref{thm:general-preservation} applied to the free $\sO$-algebra monad now completes the proof.
\end{proof}

Similarly, we can combine Theorem \ref{thm:general-preservation} and Theorem \ref{colored-model} or Theorem \ref{middle-row-one-color} to prove preservation results for colored operads that are not cofibrant in $\symseqcm$.  At this moment, we do not have easy-to-check conditions on $L_\C$ or $\M$ to guarantee that $(\spadesuit)$ or $(\clubsuit)$ are satisfied by $L_\C(\M)$. Even for the case of the commutative monoid operad, finding such a condition was difficult work in \cite{white-localization}. We hope to return to this problem in the future. For the moment, we still have preservation results if we assume $(\spadesuit)$ or $(\clubsuit)$ on both $\M$ and $L_\C(\M)$.

\begin{theorem}
\label{thm:pres-all-operads}
Suppose that $\calm$ and $L_\C(\M)$ are strongly cofibrantly generated monoidal model categories satisfying
\begin{quote}
$(\spadesuit)$ : for each $n \geq 1$ and for each object $X \in \calm^{\sigmaop_n}$, the function
\[
X \tensorover{\Sigma_n} (-)^{\boxprod n} : \calm \to \calm
\]
preserves trivial cofibrations.
\end{quote}
Then $L_\C$ preserves $\sO$-algebras for all $\fC$-colored operads $\sO$ in $\M$.
\end{theorem}

\begin{proof}
We must verify the hypotheses of Theorem \ref{thm:general-preservation}. We do so by applying Theorem \ref{colored-model} to both $\M$ and $L_\C(\M)$ to obtain transferred model structures on $\alg_\M(\sO)$ and $\alg_{L_\C(\M)}(\sO)$. That $L_\C$ preserves $\sO$-algebras then follows from Theorem \ref{thm:general-preservation}.
\end{proof}

\begin{theorem}
\label{thm:pres-levelwise-cof-operads}
Suppose that $\calm$ and $L_\C(\M)$ are cofibrantly generated monoidal model categories, and both $\M$ and $L_\C(\M)$ satisfy
\begin{quote}
$(\clubsuit)$: For each $n \geq 1$ and $X \in \calm^{\sigmaop_n}$ that is cofibrant in $\calm$, the function
\[
X \tensorover{\Sigma_n} (-)^{\boxprod n} : \calm \to \calm
\]
preserves (trivial) cofibrations.
\end{quote}
Then $L_\C$ preserves $\sO$-algebras for all entrywise cofibrant $\fC$-colored operads $\sO$ in $\M$.
\end{theorem}

\begin{proof}
We must verify the hypotheses of Theorem \ref{thm:general-preservation}. We apply Theorem \ref{middle-row-one-color} to both $\M$ and $L_\C(\M)$ to obtain transferred semi-model structures on $\alg_\M(\sO)$ and $\alg_{L_\C(\M)}(\sO)$. Note that $\sO$ is also an entrywise cofibrant $\fC$-colored operad in $L_\C(\M)$ because cofibrant objects in $\M$ and in $L_\C(\M)$ are the same. Finally, we apply \ref{thm:general-preservation} to deduce that $L_\C$ preserves $\sO$-algebras.
\end{proof}

\section{Applications} \label{sec:applications}

In this section we prove that $(\spadesuit)$ is satisfied in several model categories of interest. The results of Section \ref{sec:model-on-algebras} then provide (semi-)model structures on algebras over colored operads in these settings, and the results of Section \ref{sec:preservation} provide preservation results with respect to left Bousfield localization.

\subsection{Chain Complexes}

\begin{theorem}
Let $k$ be a field of characteristic zero and let $\Ch(k)$ denote either bounded or unbounded chain complexes over $k$. Then the projective (equivalently, the injective) model structure on $\Ch(k)$ satisfies $(\spadesuit)$.
\end{theorem}

\begin{corollary}
For any colored operad $P$ in $\Ch(k)$, $P$-algebras inherit a model structure from $\Ch(k)$.
\end{corollary}

\begin{corollary}
Any left Bousfield localization of $\Ch_{\geq 0}(k)$ preserves algebras over any colored operad $P$. Any left Bousfield localization of unbounded chain complexes that is stable (i.e. commutes with suspension) preserves algebras over any colored operad $P$.
\end{corollary}

\begin{proof}
The hypotheses of the theorem imply that all symmetric sequences in $\Ch(k)$ are projectively cofibrant, effectively by Maschke's Theorem. It follows that any $Z\in (\Ch(k))^{\Sigma_n^{op}}$ is $\Sigma_n$-projectively cofibrant. For any trivial cofibration $f$, $f^{\boxprod n}$ is a $\Sigma_n$-equivariant map which is again a trivial cofibration in $\Ch(k)$, so Lemma 2.5.2 in \cite{bm06} implies 
\[
Z\tensorover{\Sigma_n} f^{\boxprod n} = ((\emptyset \to Z)\boxprod f^{\boxprod n})_{\Sigma_n}
\]
 is a trivial cofibration in $\Ch(k)$ because $(-)_{\Sigma_n}$ is a left Quillen functor from $(\Ch(k))^{\Sigma_n^{op}}$ to $\Ch(k)$. 

The first corollary follows from Theorem \ref{colored-model}. To prove the second corollary, note that $(\spadesuit)$ remains true after any monoidal left Bousfield localization by the argument in the preceding paragraph, replacing ``trivial cofibration'' by ``locally trivial cofibration'' everywhere. The second corollary therefore follows from Theorem \ref{thm:pres-all-operads} as soon as we know the pushout product axiom holds after localization.

In $\Ch_{\geq 0}$ the only left Bousfield localizations are truncations, for which the pushout product axiom can be checked immediately (cf Corollary 7.6 in \cite{white-localization}). Unbounded chain complexes form a monogenic stable model category, and Theorem \ref{thm:PPAxiom-nontractable} demonstrates that a localization is monoidal if and only if it is stable, since tensoring with a sphere is the same thing as shifting (cf Proposition 5.4 in \cite{white-localization}). 
\end{proof}

\begin{example}
If $k$ is a field with nonzero characteristic then $\Ch(k)$ does not satisfy $(\clubsuit)_{tcof}$. Already with $X=k$, $\Sym(-)$ does not preserve trivial cofibrations, which is shown in 5.1 of \cite{white-commutative} among many other places. This is the reason commutative differential graded algebras cannot inherit a model structure from $\Ch(k)$ when char$(k)\neq 0$.
\end{example}

\subsection{Spaces}

\begin{theorem}
The Quillen model structure on (pointed or unpointed) simplicial sets $\sset$ satisfies $(\spadesuit)$. It follows that the category of algebras over any colored operad in $\sset$ inherits a model structure.
\end{theorem}

\begin{corollary}
For any colored operad $P$ in $\sset$, any left Bousfield localization preserves $P$-algebras because $(\spadesuit)$ remains true in the localized model category.
\end{corollary}

\begin{proof}
In $\sset$, the classes of cofibrations and monomorphisms coincide. It is easy to see that if $i:A\to B$ is a monomorphism then 
\[
i\boxprod i:A\times B \coprodover{A\times A} B\times A \to B\times B
\]
is a monomorphism, since $i\times A$ is a monomorphism. Similarly, for any $Z$, $Z\times i^{\boxprod n}$ is a monomorphism and the quotient of any monomorphism by a $\Sigma_n$-action is a monomorphism.

The same holds for pointed simplicial sets with the smash product, by the same argument. Here we use that $i$ preserves the base point, so again $i\wedge B$ is a monomorphism (because the subset of the domain which is contracted to the base point is taken to the subset of the codomain which is contracted). Thus, $i^{\boxprod n}$ and $Z\wedge i^{\boxprod n}$ are monomorphisms, and hence so is $Z\wedge_{\Sigma_n} i^{\boxprod n}$.

Now suppose $i$ is a trivial cofibration. Then $i^{\boxprod n}:Q_n\to B^{\times n}$ is a trivial cofibration by the pushout product axiom. For any $Z$, $Z\times i^{\boxprod n}$ is a trivial cofibration. If $Z$ has a $\Sigma_n$-action then $Z\times_{\Sigma_n} i^{\boxprod n}$ is a cofibration by the argument above and we will now prove it is a weak equivalence by the same reasoning as in Theorem 5.2 in \cite{white-commutative} (in the one-color case this has also appeared as Theorem 9.2 in \cite{casacuberta-raventos-tonks}). Lemma \ref{lemma:sym-for-sset} proves that $Z\times_{\Sigma_n} i^{\times n}$ is a weak equivalence. Lemma \ref{lemma:sym-we-iff-boxprod} proves this suffices.

The second statement of the theorem follows from Theorem \ref{colored-model}. To prove the corollary, note that every left Bousfield localization of $\sset$ is monoidal by Theorem 4.1.1 of \cite{hirschhorn}, since the pushout product axiom is the same as SM7 in this case. Next, Lemma \ref{lemma:sym-for-sset} below implies $(\spadesuit)$ remains true after localization. Finally, Theorem \ref{thm:pres-all-operads} implies the corollary. 
\end{proof}

\begin{lemma} \label{lemma:sym-we-iff-boxprod}
Let $(M,\otimes)$ be a symmetric monoidal category which is also a model category. Assume for every cofibration $i$ between cofibrant objects, every $n$, and every $Z \in \M^{\Sigma_n}$ cofibrant in $\M$, that $Z\otimes_{\Sigma_n} i^{\boxprod n}$ is a cofibration. Then $Z\otimes_{\Sigma_n}i^{\boxprod n}$ is a trivial cofibration for all $n$ if $Z\otimes_{\Sigma_n} i^{\otimes n}$ is a weak equivalence.
\end{lemma}

\begin{proof}
Let $i:X\to Y$ be a trivial cofibration between cofibrant objects, Z a cofibrant object, and assume $Z\otimes_{\Sigma_n} i^{\otimes n}$ is a weak equivalence.
Recall from (\ref{inductive-q-one-colored}) the construction of the $\Sigma_n$-equivariant maps $Q_{q-1}^n \to Q_q^n$ built from $i$ via pushouts:
\begin{align}
\label{eq:pushout_defining_equivariant_punctured_cube}
\nicexy{
  \Sigma_n \dotover{\Sigma_{n-q}\times\Sigma_{q}} X^{\otimes (n-q)} 
  \otimes Q_{q-1}^q \ar[d] \ar[r]  & Q_{q-1}^n\ar[d]  \\
  \Sigma_n\dotover{\Sigma_{n-q}\times\Sigma_{q}}X^{\otimes(n-q)}
  \otimes Y^{\otimes q}\ar[r] & Q_q^n
}
\end{align}

Observe that the pushout diagram above remains a pushout diagram if we apply $Z\otimes_{\Sigma_n}-$ to all objects and morphisms in the diagram, because $Z\otimes_{\Sigma_n}-$ is a left adjoint and so commutes with colimits. We obtain the diagram
\begin{align}
\label{eq:pushout_defining_equivariant_after_coinvariants}
\nicexy{
  Z \tensorover{\Sigma_{n-q}\times \Sigma_q} X^{\otimes (n-q)} \otimes Q_{q-1}^q \ar[d] \ar[r] & Z\tensorover{\Sigma_n} Q_{q-1}^n \ar[d]  \\
  Z\tensorover{\Sigma_{n-q}\times \Sigma_q} X^{\otimes (n-q)} \otimes Y^{\otimes q}\ar[r] & Z\tensorover{\Sigma_n}Q_q^n
}
\end{align}

Recall from (\ref{inductive-q-one-colored}) that $Q_0^n = X^{\otimes n}$ and $Q_n^n = Y^{\otimes n}$ so that the composite as $q$ varies of all the right vertical maps above is $Z\tensorover{\Sigma_n} i^{\otimes n}$, which we have assumed is a trivial cofibration. Observe that the left vertical map $Z \otimes_{\Sigma_{n-q}\times \Sigma_q} X^{\otimes (n-q)} \otimes i^{\boxprod q}$ is isomorphic to the map $W\otimes_{\Sigma_q} i^{\boxprod q}$ where $W = Z \otimes_{\Sigma_{n-q}}X^{\otimes(n-q)}$ has $\Sigma_q$ acting entirely on $Z$. Furthermore, $W$ is cofibrant in $\M$ by the hypothesis on cofibrations applied to the map $\emptyset \to X$.

We prove by induction that all maps of the form $Z\otimes_{\Sigma_n} i^{\boxprod n}$ are trivial cofibrations. For $n=1$ the map in question is $Z\otimes i$, which is a trivial cofibration because $Z$ is cofibrant so $Z\otimes -$ is a left Quillen functor. Now assume $Z\otimes_{\Sigma_k} i^{\boxprod k}$ is a trivial cofibration for all $Z$ and all $k<n$. Write $Z\otimes_{\Sigma_n} i^{\otimes n}$ as the composite:
\begin{align}
\label{eq:composite_for_Z_Qn}
Z\otimes_{\Sigma_n} Q_0^n \to \dots \to Z\tensorover{\Sigma_n} Q^n_q \to \dots \to Z\tensorover{\Sigma_n} Q^n_n
\end{align}
The inductive hypothesis and the preceding paragraph imply that the left vertical map in (\ref{eq:pushout_defining_equivariant_after_coinvariants}) is a trivial cofibration, since $W$ is cofibrant and $q<n$. Thus, the right vertical map in (\ref{eq:pushout_defining_equivariant_after_coinvariants}) is a trivial cofibration, since it is a pushout, and this implies that all maps in (\ref{eq:composite_for_Z_Qn}) are trivial cofibrations except possibly the last map $Z\otimes_{\Sigma_n} i^{\boxprod n}$. By hypothesis, this map is a cofibration. Since we have assumed the composite $Z\otimes_{\Sigma_n} i^{\otimes n}$ is a weak equivalence, the two out of three property implies the last map is a weak equivalence, hence a trivial cofibration as required. This completes the induction.
\end{proof}

\begin{lemma} \label{lemma:sym-for-sset}
Let $\mathcal{C}$ be any set of maps in $\sset$. For any $\mathcal{C}$-local equivalence $i:A\to B$, any $n$, and any simplicial set $Z$ with a $\Sigma_n$-action, $Z\times_{\Sigma_n}i^{\times n}$ is a $\mathcal{C}$-local equivalence.
\end{lemma}

\begin{proof}
As in Theorem 5.2 of \cite{white-commutative} we follow 4.A of \cite{farjoun} and consider the colimit $\colim_{H \in Orb^{op}} Z^H \times (i^{\times n})^H$ where $Orb$ is the orbit category of $\Sigma_n$ whose objects are subgroups $H < \Sigma_n$, and $(-)^H$ is the $H$-fixed points functor. This colimit is a homotopy colimit because it is indexed by a free $Orb^{op}$-diagram (and all objects are cofibrant). This colimit is isomorphic to $Z\times_{\Sigma_n}i^{\times n} = \colim_{\Sigma_n} Z\times i^{\times n}$ because the diagram for the latter is final in the diagram for the former: if two points are glued together by some $H < \Sigma_n$ then they must be glued together by the full $\Sigma_n$ action as well.

Lastly, for any simplicial set $D$ and any $H<\Sigma_n$, $(D^{\times n})^H$ is isomorphic to $D^{\times k}$ for some $k$, since $H$ acts on the factors of $D$ and hence any factor is either left fixed or identified entirely with another factor. Hence, $Z\times_{\Sigma_n} i^{\times n}$ is a homotopy colimit of maps of the form $Z^H\times i^k$. Since $\mathcal{C}$-local equivalences in sSet are closed under finite product (by 5.2 in \cite{white-localization}), each such map is a $\mathcal{C}$-local equivalence, hence their homotopy colimit is a $\mathcal{C}$-local equivalence as required. 
\end{proof}

\subsection{Symmetric Spectra}

A symmetric spectrum $X = (X_n)$ is a sequence of simplicial sets where for every $n$, $X_n$ comes equipped with an action of $\Sigma_n$, and all structure maps and morphisms respect these group actions. \cite{hovey-shipley-smith} introduced this notion and endowed the category $\Sp$ of symmetric spectra with the projective stable model structure, obtained as a left Bousfield localization of the levelwise model structure where weak equivalences (resp. (co)fibrations) are maps $f=(f_n)$ such that each $f_n$ is a weak equivalence (resp. (co)fibration) in $\sset$. There is also an injective stable model structure $\Sp_{inj}$ where the cofibrations are the monomorphisms and the weak equivalences are the stable equivalences (\cite{hovey-shipley-smith} 5.3).

An obstruction due to Gaunce Lewis demonstrates that commutative ring spectra cannot inherit a model structure from the projective stable model structure. One way to side-step this obstruction is to tweak the model structure so that the sphere spectrum is no longer cofibrant. This is accomplished by the positive stable model structure of \cite{MMSS}, which has the same weak equivalences but now cofibrations are required to have $f_0$ an isomorphism of simplicial sets. The positive flat model structure of \cite{shipley-positive} is even nicer because cofibrant commutative ring spectra forget to cofibrant spectra. This model structure is obtained by enlarging the cofibrations of symmetric sequences to equal the monomorphisms, then passing to $S$-modules where $S$ is the sphere spectrum, then left Bousfield localizing to get to the flat stable model structure, and finally requiring cofibrations to be isomorphisms in level zero.

The main result of this section is:

\begin{theorem}
\label{thm:agt1_1.1}
Endow $\Sp$ with the positive flat stable model structure or the positive stable model structure. Then $\Sp$ satisfies an extension of $(\spadesuit)$ as in Remark \ref{remark-more-gen-spadesuit}. As a consequence, for every non-empty set of colors $\fC$ and for every $\fC$-colored operad $\sO$, the category $\Alg_{\sO}$ inherits a model structure from $\Sp$.
\end{theorem}

The proof requires two lemmas. The first comes from \cite{agt1}, and would be false in non-positive model structures.

\begin{lemma}
\label{agt1_4.28*b_4.29*a}
Consider $\Sp$  with the positive flat stable model structure.  Suppose:
\begin{itemize}
\item
$t \geq 1$ and $B \in \Sp^{\sigmatop}$.
\item
$i : X \to Y \in \Sp$ is a cofibration between cofibrant objects.
\end{itemize}
Then:
\begin{enumerate}
\item
The map
\[
\nicexy@C+10pt{B \tensorover{\sigmat} Q^t_{t-1} \ar[r]^-{\Id \tensorover{\sigmat} i^{\boxprod t}}
& B \tensorover{\sigmat} Y^{\otimes t}}
\]
in $\Sp$ is a monomorphism.
\item
The functor
\[
\nicexy@C+.6cm{\Sp \ar[r]^-{B \tensorover{\sigmat} (-)^{\otimes t}} & \Sp}
\]
preserves weak equivalences between cofibrant objects.
\end{enumerate}
\end{lemma}

\begin{proof}
The first assertion is the special case of \cite{agt1} (corrigendum) Proposition 4.28$^*$(b) applied to symmetric spectra, regarded as symmetric sequences concentrated at $0$.  Likewise, the second assertion is the special case of \cite{agt1} (corrigendum) Proposition 4.29$^*$(a) applied to symmetric sequences concentrated at $0$.
\end{proof}

\begin{lemma}
\label{agt1_4.4}
Suppose $\sO$ is a $\fC$-colored operad in $\Sp$, $A \in \algo$, $i : X \to Y \in \sptothec$  is a generating acyclic cofibration with the positive flat stable model structure, and
\begin{equation}
\label{pushout-freemap-sp}
\nicexy{
\sO \circ X \ar[d]_-{\id \circ i} \ar[r]^-{f} & A \ar[d]^-{j}\\
\sO \circ Y \ar[r] & A \coprod\limits_{\sO \circ X} (\sO \circ Y)
}
\end{equation}
is a pushout in $\algo$.  Then the map $j$ is an entrywise monomorphism and weak equivalence.
\end{lemma}

\begin{proof}
In the injective model structure on $\Sp$, the cofibrations are the monomorphisms, so it is sufficient to prove each $j_t : A_{t-t} \to A_t \in \sptothec$ in our filtration (\ref{aoycolim}) of $j$ is an entrywise monomorphism and weak equivalence, since such maps are closed under pushout and transfinite composition.  Note that $i$ must be concentrated in one color, say $c \in \fC$, where it is a generating acyclic cofibration of $\Sp$ by \cite{hirschhorn} (11.1.10).  Therefore, it is enough to show that for each $t \geq 1$ and each color $d \in \fC$, the map $(j_t)_d \in \Sp$ in the pushout \eqref{one-colored-jt-pushout},
\[
\nicexy{
\sO_A\smallbinom{d}{[tc]}
\tensorover{\sigmat} 
Q^{t}_{t-1} 
\ar[d]_-{\Id \tensorover{\sigmat} i^{\boxprod t}} \ar[r]^-{f^{t-1}_*} 
& 
(A_{t-1})_d \ar[d]^-{(j_t)_d}
\\
\sO_A\smallbinom{d}{[tc]} 
\tensorover{\sigmat} Y^{\otimes t} 
\ar[r]^-{\xi_{t}} 
& 
(A_t)_d,
}\]
is a monomorphism and a weak equivalence.

Since $i \in \Sp$ is a generating acyclic cofibration with the positive flat stable model structure, it is a cofibration between cofibrant objects. Lemma \ref{agt1_4.28*b_4.29*a}(1) tells us  $\Id \otimes_{\sigmat} i^{\boxprod t}$ is a monomorphism (equivalently, a cofibration in $\Sp_{inj}$). Since cofibrations are closed under pushout, $(j_t)_d$ is a monomorphism. We now prove it is a weak equivalence.

Note that the vertical maps above are monomorphisms, so there are isomorphisms
\begin{equation}
\label{atd-over-atminusoned}
\frac{(A_t)_d}{(A_{t-1})_d}
\cong \O_A\smallbinom{d}{[tc]} \tensorover{\sigmat} \bigl(Y^{\otimes t}/Q^t_{t-1}\bigr)
\cong  \O_A\smallbinom{d}{[tc]} \tensorover{\sigmat} (Y/X)^{\otimes t}.
\end{equation}
Since $* \to Y/X \in \Sp$ is an acyclic cofibration between cofibrant objects, the above isomorphisms and Lemma \ref{agt1_4.28*b_4.29*a}(2) imply that the map
\[
\nicexy{
\ast \ar[r] 
&  \O_A\smallbinom{d}{[tc]} \tensorover{\sigmat} (Y/X)^{\otimes t}
\ar[r]^-{\cong}
& \frac{(A_t)_d}{(A_{t-1})_d}
}\] 
is a weak equivalence.  So $(j_t)_d$ is a weak equivalence.
\end{proof}

\begin{proof}[Proof of Theorem \ref{thm:agt1_1.1}]
Let $i$ be a generating trivial cofibration in the positive flat stable model structure. Since every generating trivial cofibration of the positive stable model structure is such a map, it is sufficient to consider such $i$. Remark \ref{remark-more-gen-spadesuit} demonstrates that it suffices to show that for every such $i$ and every $X\in \Sp^{\Sigma_n}$, maps of the form $X \otimes_{\Sigma_n} i^{\boxprod n}$ are contained in some class of morphisms contained in the weak equivalences and closed under transfinite composition and pushout. The class of morphisms we will use are the trivial cofibrations in the injective stable model structure on $\Sp$.

Let $\mathcal{R}$ be the closure under transfinite composition and pushout of the class of maps $X \otimes_{\Sigma_n} i^{\boxprod n}$ taken over all positive flat trivial cofibrations $i$, all symmetric spectra $X$ with a $\Sigma_n$ action, and all $n>0$. We must prove this class is contained in the weak equivalences. Let $\mathcal{R}_{inj}$ denote the same saturation, but where $i$ runs through injective trivial cofibrations. Observe that $\mathcal{R} \subset \mathcal{R}_{inj}$ because every positive flat trivial cofibration is an injective trivial cofibration. 

By Lemma \ref{agt1_4.4}, all maps in $\mathcal{R}_{inj}$ are trivial cofibrations in $\Sp_{inj}$, hence are stable equivalences. As these are also the weak equivalences of the positive (flat) stable model structure, we have proven $\Sp$ satisfies an extension of $(\spadesuit)$ as in Remark \ref{remark-more-gen-spadesuit}. In particular, transfinite compositions of maps $j$ built from $i$ as in \eqref{pushout-freemap-sp} are stable equivalences, and this proves the existence of transferred model structures on any category of colored operad algebras in $\Sp$ with either the positive stable or positive flat stable model structure, just as in the proof of Theorem \ref{colored-model}.
\end{proof}



\subsection{Further Applications}

In addition to the applications listed above, we expect this work to apply in numerous other contexts including equivariant stable homotopy theory, motivic homotopy theory, and higher category theory. Preservation results as in Section \ref{sec:preservation} have already been used in \cite{kervaire-arxiv},\cite{hill-hopkins}, and \cite{magda}. The first author hopes to apply the results in this paper--both the existence of these (semi-)model structures and their relationship to left Bousfield localization--to ongoing joint work with Javier Guti\'{e}rrez on the $N_\infty$-operads of \cite{blumberg-hill}. Similarly, a version of Theorem \ref{sigmacof-alg-semi} has been used in motivic contexts \cite{gutierrez-rondigs-spitzweck-ostvaer} and we hope the results in Section \ref{sec:alg-over-colored} can be used to weaken the cofibrancy hypotheses required there. It would also be valuable to apply the results of Section \ref{sec:preservation} in that setting. 

Theorem \ref{sigmacof-alg-semi} has already been used in \cite{batanin-stabilization} as a fundamental step in the proof of the Breen-Baez-Dolan Stabilization Hypothesis for Rezk's model of weak $n$-categories. Batanin and the first author plan to use a similar approach to prove a stronger and more general stabilization result, and also to prove Deligne's Conjecture in more general settings, extending \cite{mcclure-smith}.

In future work the authors hope to dualize Theorem \ref{thm:general-preservation} to a statement about right Bousfield localization and then to apply the (semi-)model structures to obtain results regarding preservation of algebras over colored operads under right Bousfield localization. The results in Sections \ref{sec:alg-over-colored} and \ref{sec:applications} will be crucial to this program.



\end{document}